\documentclass[a4paper,12pt]{amsart}
\usepackage{amsfonts,amsmath,amssymb,amsthm}
\usepackage{latexsym}
\usepackage{eucal}
\usepackage[dvips]{graphics}
\usepackage{mathabx}
\usepackage[english]{babel}
\usepackage{epsfig}
\usepackage{enumerate}
\usepackage{float}
\usepackage[usenames]{color}
\usepackage[colorlinks=true, linkcolor=red, citecolor=blue]{hyperref}
\newcommand{\nwc}{\newcommand}
\nwc\eps{\varepsilon}

\newtheorem{theorem}{Theorem}[section]
\newtheorem{proposition}[theorem]{Proposition}

\newtheorem{claim}{Claim}
\newtheorem{remark}{Remark}
\theoremstyle{remark}

\usepackage{mathrsfs}

\numberwithin{equation}{section}
\makeatletter
\@namedef{subjclassname@2010}{\textup{2020} Mathematics Subject Classification}
\makeatother

\begin{document}

\title[Stabilization of a KdV-KdV system with delay]{On the boundary stabilization of the KdV-KdV system with time-dependent delay}
% {A note about the stabilization of the Boussinesq KdV-type system using the time-varying delay feedback}

%----------Author 1
\author[Capistrano-Filho]{Roberto de A.  Capistrano--Filho}
\address{
Departamento de Matem\'atica, Universidade Federal de Pernambuco\\
Cidade Universit\'aria, 50740-545, Recife (PE), Brazil\\
Email address: \normalfont\texttt{roberto.capistranofilho@ufpe.br}}

%\email{roberto.capistranofilho@ufpe.br}

%\thanks{This work was completed with the support of our
%----------Author 2
\author[Chentouf]{Boumedi\`ene Chentouf*}
\thanks{*Corresponding author.}
\address{
Faculty of Science, Kuwait University \\
Department of Mathematics, Safat 13060, Kuwait\\
Email address: \normalfont\texttt{boumediene.chentouf@ku.edu.kw}}

%\email{jose.quintero@correounivalle.edu.co}
%----------Author 3

\author[Gonzalez Martinez]{Victor H. Gonzalez Martinez}
\address{
Departamento de Matem\'atica, Universidade Federal de Pernambuco\\
Cidade Universit\'aria, 50740-545, Recife (PE), Brazil\\
Email address: \normalfont\texttt{victor.martinez@ufpe.br}}
%\email{jose.quintero@correounivalle.edu.co}

%----------Author 4

\author[Muñoz]{Juan Ricardo Muñoz}
\address{
Departamento de Matem\'atica, Universidade Federal de Pernambuco\\
Cidade Universit\'aria, 50740-545, Recife (PE), Brazil\\
Email address: \normalfont\texttt{juan.ricardo@ufpe.br}}



\date{}
%\dedicatory{xxx]
%%% ----------------------------------------------------------------------

\numberwithin{equation}{section}

\begin{abstract}
The boundary stabilization problem of the Boussinesq KdV-KdV type system is investigated in this paper. An appropriate boundary feedback law consisting of a linear combination of a damping mechanism and a delay term is designed. Then, first, considering time-varying delay feedback together with a smallness restriction on the length of the spatial domain and the initial data, we show that the problem under consideration is well-posed. The proof combines Kato's approach and the fixed-point argument. Last but not least, we prove that the energy of the linearized KdV-KdV system decays exponentially by employing the Lyapunov method.
\end{abstract}

\maketitle

\section{Introduction}
\subsection{Boussinesq system model}
The Boussinesq system is a set of partial differential equations (PDEs) that describe the behavior of waves in fluids with small-amplitude and long-wavelength disturbances. It was first introduced by the French mathematician Joseph Boussinesq in the 19th century as a way to model waves in shallow water \cite{bou}. Since then, the system has been used to study a wide range of physical phenomena, including ocean currents, atmospheric circulation, and heat transfer in fluids. The Boussinesq system is also an important tool in the study of fluid dynamics and has applications in a variety of fields, including meteorology, oceanography, and engineering.

Recently, Bona \textit{et al.} in \cite{Bona2002,Bona2004} developed a four-parameter family of Boussinesq systems to describe the motion of small-amplitude long waves on the surface of an ideal fluid under gravity and in situations where the motion is sensibly two-dimensional. They specifically investigated a family of systems of the form
\begin{equation}\label{eq:Boussi}
\begin{cases}
\eta_t(t,x)+\omega_x(t,x)+a\omega_{xxx}(t,x)-b\eta_{xxt}(t,x) + (\eta(t,x)\omega(t,x))_x = 0, \\
\omega_t(t,x)+\eta_x(t,x)+c\eta_{xxx}(t,x)-d\omega_{xxt}(t,x)+\omega(t,x)\omega_x(t,x) = 0,
\end{cases}
\end{equation}
which are all Euler equation approximations of the same order. Here $\eta$ represents the elevation of the equilibrium point and $\omega = \omega_\theta$ is the horizontal velocity in the flow at height $\theta\ell$, where $\theta \in [0,1]$ and $\ell$ is the undisturbed depth of the fluid. The parameters $a, b, c, d$, that one might choose in a given modeling situation, are required to fulfill the relations
$ a+b = \frac{1}{2}\left( \theta^2-\frac{1}{3} \right)$ and $ c+d = \frac{1}{2}(1-\theta^2) \geq 0.$

When $b = d = 0$ and making a scaling argument, we obtain the Boussinesq system of KdV-KdV type
\begin{equation}\label{eq:KdV-KdVa}
\begin{cases}
\eta_t(t,x)+\omega_x(t,x)+\omega_{xxx}(t,x) + (\eta(t,x)\omega(t,x))_x = 0, \\
\omega_t(t,x)+\eta_x(t,x)+\eta_{xxx}(t,x)+\omega(t,x)\omega_x(t,x) = 0,
\end{cases}
\end{equation}
which is shown to admit global solutions on $\mathbb{R}$ and also has good control properties such as stabilization, and controllability, in periodic framework $\mathbb{T}$\footnote{See \cite{Bona2004} for the real-line case and \cite{RoAn,Micu2009} for details in the periodic framework.}. Nonetheless, stabilization properties for the Boussinesq KdV-KdV system on a bounded domain of $\mathbb{R}$ is a challenging problem due to the coupling of the nonlinear and dispersive nature of the PDEs. In this spirit, a few works indicate that appropriate boundary feedback controls provide good stabilization results to the system \eqref{eq:KdV-KdVa} on a bounded domain $\mathbb{R}$ (see, for instance, \cite{Cerpa20XX,Capistrano2018,Capistrano2019,Pazoto2008}). To be more precise, in  \cite{Pazoto2008}, a set of boundary controls is needed so that the solutions of the system \eqref{eq:KdV-KdVa} issuing from small data globally exist and the corresponding energy exponentially decay. Indeed, \eqref{eq:KdV-KdVa} is coupled with the following boundary conditions:
\[
\begin{cases}
\omega(t,0) = \omega_{xx}(t,0) =0, & t>0, \\
\omega_x(t,0) =a_0 \eta_x(t,0), & t>0, \\
\omega_x(t,L) =-a_1 \eta_x(t,L), & t>0, \\
\omega_x(t,L) = \eta_x(t,L), \;\;
\omega_{xx}(t,L) = - \eta_{xx}(t,L), & t>0,
\end{cases}
\]
where $a_0 \ge 0$, whereas $a_1>0$.
Later, two boundary controls are designed via the backstepping method to obtain a local rapid exponential stabilization result for the solutions to \eqref{eq:KdV-KdVa} \cite{Capistrano2018}. In turn, the main concern in \cite{Capistrano2019} is the exact controllability of \eqref{eq:KdV-KdVa}. Specifically, a control of Neumann type is proposed to reach a local exact controllability property as well as the exponential stability of the system. Lastly, the linear variant of \eqref{eq:KdV-KdVa} is considered and a single linear boundary control is designed to obtain the rapid stabilization of the solutions \cite{Cerpa20XX}.

\subsection{Problem setting} First, let us consider the KdV-KdV equation \eqref{eq:KdV-KdVa} but in a bounded domain $[0,L]$ and with the following set of boundary conditions
\begin{equation}\label{eq:BCondR}
\begin{cases}
\eta(t,0) = \eta(t,L) = \eta_x(t,0) =0, & t\in\mathbb{R}^{+}, \\
 \omega(t,0) = \omega(t,L)= \omega_x(t,L) =0, & t\in\mathbb{R}^{+}.
\end{cases}
\end{equation}
As mentioned before, note that considering the system described above, two important facts need to be mentioned:

\vspace{0.1cm}

$\bullet$ We first notice that the \textit{global Kato smoothing effect} does not hold for the set of boundary condition \eqref{eq:BCondR}. This makes impossible the task of showing the well-posedness findings by employing classical methods, such as semigroup theory, and hence the well-posedness problem of this system remains open.

\vspace{0.1cm}

$\bullet$ The second issue is related to the energy of the system  \eqref{eq:KdV-KdVa} and \eqref{eq:BCondR}. Under the above boundary conditions, a simple integration by parts yields
$$
\frac{d}{dt}E_0 (t)
=-\int_0^L(\eta(t,x)\omega(t,x))_{x}\eta(t,x)\,dx,
$$
where $$E_0(t) = \frac{1}{2} \int_0^L ( \eta^2 (t,x)+ \omega^2(t,x) )\,dx$$
is the total energy associated with \eqref{eq:KdV-KdVa} and \eqref{eq:BCondR}. This indicates that \textit{we do not have any control over the energy in the sense that its time derivative does not have a fixed sign}.

\vspace{0.1cm}

Therefore, due to the restriction presented in these two points, the following questions naturally arise:

\vspace{0.1cm}

\noindent\textbf{Question $\mathcal{A}$}: \textit{Is there a suitable set of boundary conditions so that the Kato smoothing effect can be revealed?}

 \vspace{0.1cm}

\noindent \textbf{Question $\mathcal{B}$}: \textit{Is there a feedback control law that permits the control of the nonlinear term presented in the derivative of the energy associated with the closed-loop system? Moreover, is this desired feedback law strong enough in the presence of a time-dependent delay?}

 \vspace{0.1cm}

\noindent \textbf{Question $\mathcal{C}$}:  \textit{If the answer to these previous questions is yes, does $E_0(t)\to0$ as $t\to\infty$? If this is the case, can we give an explicit decay rate?}

 \vspace{0.1cm}

Our motivation in this work is to give answers to these questions. In this spirit, and to deal with the Boussinesq system of KdV-KdV type~\eqref{eq:KdV-KdVa}, let us consider the set of boundary conditions:
\begin{equation}\label{eq:BCond}
\begin{cases}
\eta(t,0) = \eta(t,L) = \eta_x(t,0) = \omega(t,0) = \omega(t,L)= 0, & t>0, \\
\omega_x(t,L) = -\alpha \eta_x(t,L)+\beta \eta_x(t-\tau(t),L), & t>0,
\end{cases}
\end{equation}
where $\tau(t)$ is the time-varying delay, while $\alpha$ and $\beta$ are feedback gains.

\begin{remark} The following remarks are now in order.
\begin{itemize}
\item[i.] Note that our new set of boundary conditions contains a damping mechanism $\alpha \eta_x(t,L)$ as well as the time-varying delayed feedback $\beta \eta_x(t-\tau(t),L)$.
\item[ii.] The damping mechanism will guarantee the Kato smoothing effect, which is paramount to proving the well-posedness of the system under consideration in this article.
\item[iii.] The time-varying delay feedback, together with the damping mechanism, permits to drive the energy to $0$, as $t$ goes to $\infty$, giving the stabilization of the system \eqref{eq:KdV-KdVa} and \eqref{eq:BCond}, with a precise decay rate.
\item[iv.] We point out that our main result, given in the next subsection, ensures the exponential stability of the linearized system associated with \eqref{eq:KdV-KdVa}--\eqref{eq:BCond} employing $\tau(t)$ as a time-varying delay. However, due to the lack of a priori $L^2$-estimate, it is hard to extend the result to the nonlinear system \eqref{eq:KdV-KdVa}--\eqref{eq:BCond}. We instruct the reader to see the discussion about this point in Section \ref{sec4}.
\end{itemize}
\end{remark}

It is also noteworthy that the time-delay phenomenon is practically unavoidable because of miscellaneous reasons. Indeed, it often occurs in numerous areas such as biology, mechanics, and engineering due to the dynamics of the actuators and sensors. Having said that, there is in literature a predominant opinion that time delay has intrinsically a disadvantage on the performance of practical systems (see for instance the first papers that treated this subject in the PDEs framework \cite{d1,d2,d3}). This gives rise to a monumental endeavor in attempting to nullify any negative impact of the presence of a delay on a system. In fact, the authors in \cite{Nicaise2006,xyl} show that the solutions to the wave equation remain stable provided that the delayed term is small, otherwise the stability property is lost. This outcome is extended in \cite{np1} to a general class of second-order evolution equations with unbounded time-dependent delayed control. Similar results are also obtained for numerous systems with time-dependent delay (see for instance \cite{np2,np3,Nicaise2009} and the references therein). Note also that in the context of dispersive equations, time-delayed feedback is a challenging problem as it can lead to instability or oscillatory behavior in numerous instances. Some recent articles - not exhaustive -  already addressed the stabilization problem of dispersive systems with delay. We can cite, for example, \cite{Valein2019}, \cite{bc2021} and \cite{Boumediene2023} for KdV, KS, and Kawahara equations, where time-delay boundary controls are considered. Furthermore, if the time delay occurs in the equation, the authors in  \cite{Valein20XX}, \cite{Martinez2022,Chentouf22}, and \cite{Munoz2022} showed stabilization results for the KdV, fifth-order KdV, and Kawahara-Kadomtsev-Petviashvili equations, respectively. Finally, we point out that using the time-varying delay, the authors in  \cite{Parada2022} obtained stabilization outcomes for the KdV equation. To the author's best knowledge, this is the only work that considers a coupled dispersive system with a time-dependent delay and we believe that the techniques presented here can be adapted to other systems.

\subsection{Main results and paper's outline} To our knowledge, due to the previous restrictions, there is no result combining the damping mechanism and the boundary time-varying delay to guarantee stabilization results for the linearized KdV-KdV system associated with \eqref{eq:KdV-KdVa}--\eqref{eq:BCond}. In order to state the main result and provide answers to the questions previously mentioned, we assume that there exist two positive constants $M$ and $d<1$ such that the time-dependent function $\tau(t)$ satisfies the following standard conditions:
\begin{equation}\label{eq:TauCond}
\begin{cases}
0 < \tau(0) \leq \tau(t) \leq M, \quad
\dot\tau(t)\leq d < 1,&\forall t\geq 0,\\
\tau\in W^{2, \infty}([0,T]),&T>0.
\end{cases}
\end{equation}
Furthermore, the feedback gains $\alpha$ and $\beta$  must satisfy the following constraint
\begin{equation}\label{eq:CCond}
(2\alpha-\lvert \beta \rvert)(1-d)>\lvert \beta \rvert, \ \ \text{for $0\leq d <1$}.
\end{equation}
or equivalently,
\[
\alpha>\frac{\lvert\beta\rvert}{2}\left(\frac{2-d}{1-d}\right), \ \ \text{for $0\leq d <1$}.
\]

Next, let $X_{0}:= L^2(0,L)\times L^2(0,L),$ $ H:=X_0 \times L^2(0,1)$ and consider the space
$$\mathcal{B}:=C([0, T ], X_0) \cap L^2 (0, T, [H^1(0,L)]^2),$$
whose norm is
$$
\lVert (\eta,\omega) \rVert_{\mathcal{B}} = \sup_{t\in[0,T]} \lVert (\eta(t),\omega(t)) \rVert_{X_0}
+\lVert (\eta_x,\omega_x) \rVert_{L^2(0,T; X_0)}.$$
Whereupon,  we are interested in the behavior of the solutions of the system
 \begin{equation}\label{eq:KdV-KdV}
\begin{cases}
\eta_t(t,x)+\omega_x(t,x)+\omega_{xxx}(t,x) + (\eta(t,x)\omega(t,x))_x = 0,&\mathbb{R}^{+}\times(0,L), \\
\omega_t(t,x)+\eta_x(t,x)+\eta_{xxx}(t,x)+\omega(t,x)\omega_x(t,x) = 0, &\mathbb{R}^{+}\times(0,L), \\
\eta(t,0) = \eta(t,L) = \eta_x(t,0) = \omega(t,0) = \omega(t,L)= 0, & t\in\mathbb{R}^{+},\\
\omega_x(t,L) = -\alpha \eta_x(t,L)+\beta \eta_x(t-\tau(t),L), & t>0,\\
\eta_x(t-\tau(0),L) = z_0(t-\tau(0))\in L^2(0,1) , & 0<t<\tau(0), \\
\left(\eta(0,x),\omega(0,x)\right) = \left(\eta_0(x),\omega_0(x)\right)\in X_0, & x\in(0,L).
\end{cases}
\end{equation}

It is noteworthy that the total energy associated with the system \eqref{eq:KdV-KdV} will be defined in $H$ by
\begin{equation}\label{eq:En}
E(t) = \frac{1}{2} \int_0^L(  \eta^2(t,x) +  \omega^2(t,x) )\,dx
+ \frac{\lvert \beta\rvert }{2}\tau(t) \int_0^1 \eta_x^2(t-\tau(t)\rho , L)\,d\rho.
\end{equation}
Thereafter, the principal result of the article ensures that the energy $E(t)$ decays exponentially despite the presence of the delay. An estimate of the decay rate is also provided. This answers each question that we tabled previously.
\begin{theorem}\label{th:Lyapunov0}
Let $0<L<\sqrt{3}\pi$. Suppose that~\eqref{eq:TauCond} and \eqref{eq:CCond} are satisfied. Then, for two positive constants $\mu_1$ and $\mu_2 $ with $ \mu_1 L < 1$, there exist
\begin{equation}\label{eq:r}
\zeta = \frac{1 + \max\{\mu_1 L,\mu_2\}}{1- \max\{\mu_1 L,\mu_2\}},
\end{equation}
and
\begin{equation}\label{eq:lambda}
 \lambda \leq \min \left\lbrace
\frac{\mu_1(3\pi^2-L^2)}{L^2(1+\mu_1)}  , \frac{\mu_2(1-d)}{M(1+\mu_2)}
\right\rbrace
\end{equation}
such that the energy $E(t)$ given by \eqref{eq:En} associated to the linearized  system of \eqref{eq:KdV-KdV} around the origin satisfies
$$
E(t) \leq \zeta E(0)e^{-\lambda t}, \quad \hbox{ for all } t \geq 0.$$
 \end{theorem}
 
 This outcome brings a new contribution of the stability of the KdV-KdV system with a delay term since in \cite{Cerpa20XX,Capistrano2018,Capistrano2019,Pazoto2008} no delay was considered. Moreover, unlike these papers, the spectral analysis of the linearized system cannot be conducted due to the time dependency of the delay. In turn, this prevents us from getting the set of critical lengths. The approach used in the current work is direct as it is based on the Lyapunov method.

\vspace{0.2cm}

We end this section by providing an outline of this paper, which consists of four parts including the Introduction. Section \ref{sec2} discusses the existence of local solutions for the nonlinear Boussinesq KdV-KdV system \eqref{eq:KdV-KdV}. Section \ref{sec3} is devoted to proving the stabilization result, Theorem \ref{th:Lyapunov0}, for the linearized system associated with \eqref{eq:KdV-KdV}. Additionally, we have shown that the decay rate $\lambda$ of Theorem \ref{th:Lyapunov0} can be optimized. Finally, in Section \ref{sec4}, we will provide some concluding remarks and discuss open problems related to the stabilization of the nonlinear Boussinesq KdV-KdV system \eqref{eq:KdV-KdV}.

 \section{Well-posedness theory}\label{sec2}

%We shall first investigate the well-posedness of system~\eqref{eq:KdV-KdV} in the space $\mathcal{B}$.
\subsection{Linear problem}\label{ssec:WPlin}
Consider the following linear Cauchy problem
\begin{equation}\label{eq:Cauchy}
\begin{split}
\dfrac{d}{dt} U(t) = A(t)U(t),\\
U(0) = U_0, \quad t>0,
\end{split}
\end{equation}
where $A(t)\colon D(A(t))\subset {H}\to {H}$ is densely defined. If $D(A(t))$ is independent of time $t$, i.e., $D(A(t)) = D(A(0)),$ for $t > 0.$ The next theorem ensures the existence and uniqueness of the Cauchy problem \eqref{eq:Cauchy}.
\begin{theorem}[\cite{Kato1970}]\label{th:KatoCauchy}
Assume that:
\begin{enumerate}
\item $\mathcal{Z} = D(A(0))$ is a dense subset of $H$ and $ D(A(t)) = D(A(0))$, for all $t > 0$,
\item %for all $t \in [0, T ]$,
$A(t)$ generates a strongly continuous semigroup on $H$. Moreover, the family $\{A(t) \colon t\in [0, T ]\}$ is stable with stability constants $C, \ m$ independent of $t$.% (i.e. the semigroup $(S_t(s))_{s\geq 0}$ generated by $A(t)$ satisfies $\lVert S_t (s)U \rVert_{H} \leq C e^{ms}\lVert U\rVert_H$ , for all $U \in H$ and $s \geq  0$),
\item $\partial_t A(t)$ belongs to $L_{\ast}^\infty([0, T ], B(\mathcal{Z}, H))$, the space of equivalent classes of essentially bounded, strongly measure functions from $[0, T ]$ into the set $B(\mathcal{Z}, H)$ of bounded operators from $\mathcal{Z}$ into $H$.
\end{enumerate}
Then, problem~\eqref{eq:Cauchy} has a unique solution $U \in C([0, T ], \mathcal{Z}) \cap C^1 ([0, T ], H)$ for any initial data in $\mathcal{Z}$.
\end{theorem}

The task ahead is to apply the above result to ensure the existence of solutions for the linear system associated with \eqref{eq:KdV-KdV}. To do that, consider the following linearized system associated with \eqref{eq:KdV-KdV}, that is, consider the equation without $\omega(t,x)\omega_x(t,x)$ and $(\eta(t,x)\omega(t,x))_x$.
%\begin{equation}\label{eq:KKlin}
%\begin{cases}
%\eta_t(t,x)+\omega_x(t,x)+\omega_{xxx}(t,x) = 0, & t>0, x\in(0,L), \\
%\omega_t(t,x)+\eta_x(t,x)+\eta_{xxx}(t,x) = 0, & t>0, x\in(0,L),  \\
%\eta(t,0) = \eta(t,L) = \eta_x(t,0) = \omega(t,0) = \omega(t,L)= 0, & t>0, \\
%\omega_x(t,L) = -\alpha \eta_x(t,L)+\beta \eta_x(t-\tau(t),L), & t>0, \\
%\eta_x(t-\tau(0),L) = z_0(t-\tau(0)), & 0<t<\tau(0), \\
%\left(\eta(0,x),\omega(0,x)\right) = \left(\eta_0(x),\omega_0(x)\right), & x\in(0,L).
%\end{cases}
%\end{equation}
Following the ideas introduced by in \cite{xyl, Nicaise2006, Nicaise2009}, let us define the auxiliary variable
$$z(t,\rho) = \eta_x(t- \tau(t)\rho,L),$$
which satisfies the transport equation:
\begin{equation}\label{eq:tr}
\begin{cases}
\tau(t)z_t(t,\rho) + (1-\dot\tau(t)\rho)z_\rho(t,\rho) = 0, &t>0, \rho \in(0,1), \\
z(t,0) = \eta_x(t,L), \ z(0,\rho) = z_0(-\tau(0)\rho) & t>0, \  \rho \in (0,1).
\end{cases}
\end{equation}
Now, the space $H$ will be equipped with the inner product
\begin{equation}\label{inner}
\begin{aligned}
\left\langle \left(\eta,\omega,z\right), \left(\tilde\eta, \tilde\omega, \tilde{z}\right)
\right\rangle_{t}
=\ &
\left\langle\left(\eta, \omega \right),\left(	\tilde\eta , \tilde\omega\right)
\right\rangle_{X_0}
+ \lvert \beta\rvert  \tau(t)
\left\langle
	z, \tilde z
\right\rangle_{L^2(0,1)} ,
\end{aligned}
\end{equation}
for any $(\eta,\omega;z), (\tilde\eta,\tilde\omega;\tilde z)\in H$.

Now, we pick up $U = (\eta,\omega; z)^{T}$ and consider the time-dependent operator
$$A(t)\colon D(A(t))\subset {H}\to {H}$$ given by
\begin{equation}\label{eq:A}
A(t)\left(\eta, \omega , z\right) :=\left(-\omega_x-\omega_{xxx}, -\eta_x-\eta_{xxx}, \dfrac{\dot\tau(t)\rho -1}{\tau(t)} z_\rho\right),
\end{equation}
with domain defined by
\begin{equation}\label{DAA}
D(A(t))=
\left\lbrace
	\begin{aligned}
		(\eta,\omega) &\in \left[ H^3(0,L)\cap H_0^1(0,L) \right]^2, \\
		z & \in H^1(0,1),
	\end{aligned}
	\left\lvert
	\begin{aligned}
	&\eta_x(0)=0, z(0) = \eta_x(L),\\
	& \omega_x(L) = -\alpha \eta_x(L) + \beta z(1)
	\end{aligned}
	\right.
\right\rbrace.
\end{equation}
Whereupon, we rewrite~\eqref{eq:tr}-\eqref{DAA} as an abstract Cauchy problem \eqref{eq:Cauchy}. Moreover, note that $D(A(t))$ is independent of time $t$ since $D(A(t)) = D(A(0))$.

Subsequently, consider the triplet $\lbrace A, H, \mathcal{Z}\rbrace$, with $A = \left\lbrace A(t)\colon t\in[0,T] \right\rbrace$ for some $T>0$ fixed and $\mathcal{Z} = D(A(0))$.  Now, we can prove a well-posedness result of \eqref{eq:Cauchy} related to $\lbrace A, H, \mathcal{Z}\rbrace$.
\begin{theorem}
Assume that $\alpha$ and $\beta$ are real constants such that  \eqref {eq:CCond} holds. Taking $U_0 \in H$,  there exists a unique solution $U \in C([0, +\infty), H)$ to~\eqref{eq:Cauchy}. Moreover, if $U_0  \in D(A(0))$, then $U \in C([0, +\infty), D(A(0)))\cap C^1([0, +\infty), H).$
\end{theorem}
\begin{proof}
The result will be proved classically (see, for instance, \cite{Nicaise2009}). First, it is not difficult to see that $\mathcal{Z} = D(A(0))$ is a dense subset of $H$ and $D(A(t))=D(A(0))$, for all $t>0$. Therefore, the condition (1) of Theorem \ref{th:KatoCauchy} holds.

For the requirement (2) of Theorem~\ref{th:KatoCauchy}, observe that integrating by parts and using the boundary conditions, we have that
\begin{equation*}
\left\langle A(t) U, U \right\rangle_t - \kappa(t) \left\langle U, U \right\rangle_t
\leq \frac{1}{2} \left(\eta_x(L), \eta_x(t-\tau(t),L) \right)
\Phi_{\alpha,\beta}
\left(	\eta_x(L), \eta_x(t-\tau(t),L) \right)^T
\end{equation*}
where
\begin{equation*}
\kappa(t) = \frac{(\dot\tau(t)^2+1)^{\frac12}}{2\tau(t)} \quad \text{and}\quad
\Phi_{\alpha,\beta} = \begin{pmatrix}
-2\alpha + \lvert\beta\rvert & \beta\\
\beta & \lvert\beta\rvert (d -1 )
\end{pmatrix}.
\end{equation*}
Invoking \eqref {eq:CCond}, we deduce that $\Phi_{\alpha,\beta}$ is a negative definite matrix and consequently we get
$$\left\langle A(t) U, U \right\rangle_t - \kappa(t) \left\langle U,U \right\rangle_t \leq 0.$$
Thereby, $\tilde{A}(t) = A(t) - \kappa(t)I$ is dissipative.

On the other hand, we claim the following:
\begin{claim}\label{CL1}
For all $t\in[0,T]$, the operator $A(t)$ is maximal, or equivalently, we have that $\lambda I - A(t)$ is surjective,  for some  $\lambda > 0$.
\end{claim}

In fact, let us fix $t\in[0,T]$. Given $(f_1,f_2, h)^T \in H$, we look for  $U = (\eta,\omega, z )^T \in D(A(t))$ solution of
%, that is,
\begin{equation}\label{eq:WP_ewz}
(\lambda I - A(t))U = (f_1, f_2, h) \iff \begin{cases}
\lambda\eta + \omega_x + \omega_{xxx} = f_1, \\
\lambda\omega + \eta_x + \eta_{xxx} = f_2, \\
\lambda z + \left(\dfrac{1-\dot\tau(t)\rho}{\tau(t)}\right) z_\rho = h,\\
\eta(0) = \eta(L) = \omega(0) = \omega(L) = \eta_x(0)=0, \\
\omega_x(L) = -\alpha \eta_x(L) + \beta z(1), z(0) = \eta_x(L).
\end{cases}
\end{equation}
A straightforward computation gives that $z$ has the explicit form
%satisfies the following ordinary differential equation
%\begin{equation*}\begin{cases}
%\lambda z + \left(\dfrac{1-\dot\tau(t)\rho}{\tau(t)}\right) z_\rho = h, \\
%z(0) = \eta_x(L),
%\end{cases}
%\end{equation*}
\begin{equation*}
z(\rho) =
\begin{cases}\displaystyle
\eta_x(L) e^{-\lambda\tau(t)\rho} +\displaystyle \tau(t) e^{-\lambda\tau(t)\rho} \int_0^\rho e^{\lambda\tau(t)\sigma} h (\sigma)\,d\sigma, &\!\!\!\text{if }\dot\tau(t) = 0, \\[1mm]
\displaystyle
e^{\lambda\frac{\tau(t)}{\dot\tau(t)}\ln(1-\dot\tau(t)\rho)} \left[ \eta_x(L)
\displaystyle+
\int_0^\rho \frac{h(\sigma)\tau(t)}{1-\dot\tau(t)\sigma}  e^{-\lambda\frac{\tau(t)}{\dot\tau(t)}\ln(1-\dot\tau(t)\sigma)}\,d\sigma \right], &\!\!\! \text{if }\dot\tau(t) \neq 0.
\end{cases}
\end{equation*}
In particular, $z(1) = \eta_x(L) g_0(t) + g_h(t)$, where
\begin{equation*}
g_0(t) =
\begin{cases}\displaystyle
e^{-\lambda\tau(t)}, &\text { if } \dot\tau(t) = 0,\\ \displaystyle
e^{\lambda\frac{\tau(t)}{\dot\tau(t)}\ln(1-\dot\tau(t))}, & \text { if }\dot\tau(t) \neq 0,
\end{cases}
\end{equation*}
and
\begin{equation*}
g_h(t)=
\begin{cases}\displaystyle
 \tau(t) e^{-\lambda \tau(t)} \int_0^1 e^{\lambda \tau(t) \sigma} h(\sigma) d \sigma, & \text { if } \dot{\tau}(t)=0, \\ \displaystyle
e^{\lambda \frac{\tau(t)}{\dot\tau(t)} \ln(1-\dot{\tau}(t))} \int_0^1 \frac{h(\sigma) \tau(t)}{1-\dot{\tau}(t) \sigma} e^{-\lambda \frac{\tau(t)}{\dot{\tau}(t)} \ln (1-\dot{\tau}(t) \sigma)} d \sigma, & \text { if } \dot{\tau}(t) \neq 0.
\end{cases}
\end{equation*}
This, together with \eqref{eq:WP_ewz}, leads to claim that  $\eta$ and $\omega$ should satisfy
\begin{equation}\label{eq:WP_ew1}
\begin{cases}
\lambda\eta + \omega_x + \omega_{xxx} = f_1, \\
\lambda\omega + \eta_x + \eta_{xxx} = f_2,
\end{cases}
\end{equation}
with boundary conditions
\begin{equation}\label{eq:WP_ew1a}
\begin{cases}
\eta(0) = \eta(L) = \omega(0) = \omega(L) = \eta_x(0)=0, \\
\omega_x(L) = (-\alpha+\beta g_0(t)) \eta_x(L) + \beta g_h(t).
\end{cases}
\end{equation}
Pick $\psi(x,t) = \dfrac{x(x-L)}{L}\beta g_h(t) \in C^\infty([0,L])$ and let $\hat{\omega}:=\omega - \psi$. Then, the system \eqref{eq:WP_ew1} can be rewritten as follows:
\begin{equation}\label{eq:WP_ew2}
\begin{cases}
\lambda\eta + \omega_x + \omega_{xxx} = f_1-\psi_x =: \tilde{f_1}, \\
\lambda\omega + \eta_x + \eta_{xxx} = f_2- \lambda\psi =: \tilde{f_2},
\end{cases}
\end{equation}
and must be coupled with \eqref{eq:WP_ew1a}. Here, let us mention that for the sake of presentation clarity, we still use $\omega$ after translation. One can check that $0<g_0(t)<1$. Indeed, if $\dot{\tau}(t)=0$, then we clearly have $0<g_0(t)<1$. In turn, if $\dot{\tau}(t)\neq0$, then we have two cases to consider, namely $0<\dot{\tau}(t)<1$ and $\dot{\tau}(t)<0$. In the first case, we have
$\ln(1-\dot{\tau}(t))<\ln(1)=0$ and $\lambda\tau(t)/\dot{\tau}(t)>0,$ which implies that $0<g_0(t)=e^{\lambda\frac{\tau(t)}{\dot\tau(t)}\ln(1-\dot\tau(t))}<e^0=1.$ In the second case, we have $\ln(1-\dot{\tau}(t))>\ln(1)=0$ and $\lambda\tau(t)/\dot{\tau}(t)<0,$ which ensures that $0<g_0(t)=e^{\lambda\frac{\tau(t)}{\dot\tau(t)}\ln(1-\dot\tau(t))}<e^0=1.$
We infer from this discussion that  $-\tilde\alpha := -\alpha + \beta g_0(t) < 0$, thanks to \eqref {eq:CCond}. Thereby, our Claim \ref{CL1} is reduced to proving that $\lambda I - \hat{A}$ is surjective, where $\hat{A}$ is given by
$$
\hat{A} (\eta, \omega) = (-\omega_x-\omega_{xxx}, -\eta_x-\eta_{xxx}),
$$
while its dense domain is
$$
D(\hat{A}) := \left\lbrace
(\eta,\omega)\in \left[H^3(0,L)\cap H_0^1(0,L)\right]^2
\colon
\eta_x(0) = 0,\ \omega_x(L) = -\tilde\alpha \eta_x(L)
\right\rbrace \subset X_0.
$$
Thanks to \cite[Proposition~4.1]{Capistrano2019}, the operators $\hat{A}$ and $\hat{A}^\ast$ are dissipative, and the desired result follows by Lummer-Phillips Theorem (see, for example, \cite{Pazy}). This shows the Claim \ref{CL1}. Consequently,  $\tilde{A}(t)$ generates a strongly semigroup on $H$ and $\tilde{A} = \{\tilde{A}(t), t \in [0, T ]\}$ is a stable family of generators in $H$ with a stability constant independent of $t$, and hence the condition (2) of Theorem \ref{th:KatoCauchy} is satisfied.

Finally, due to the fact that $\tau \in W^{2,\infty}([0, T ])$ for all $T>0$, we have
$$
\dot{\kappa}(t)=\frac{\ddot{\tau}(t) \dot{\tau}(t)}{2 \tau(t)\left(\dot{\tau}(t)^2+1\right)^{1 / 2}}-\frac{\dot{\tau}(t)\left(\dot{\tau}(t)^2+1\right)^{1 / 2}}{2 \tau(t)^2}
$$
is bounded on $[0, T]$ for all $T>0$. Moreover,
\begin{equation*}
\frac{d}{dt} A(t) U=\left(0, 0 ,\frac{\ddot{\tau}(t) \tau(t) \rho-\dot{\tau}(t)(\dot{\tau}(t) \rho-1)}{\tau(t)^2} z_\rho\right),
\end{equation*}
while the coefficient of $z_\rho$ is bounded on $[0, T]$ and the regularity (3) of Theorem~\ref{th:KatoCauchy} is fulfilled.

As a consequence, all the assumptions of Theorem~\ref{th:KatoCauchy} are verified. Therefore, for $U_0 \in D({A}(0))$, the Cauchy problem
\begin{equation*}
\begin{split}
\tilde{U}_t(t)=\tilde{A}(t) \tilde{U}(t), \ \tilde{U}(0)=U_0,&\quad t>0,
\end{split}
\end{equation*}
has a unique solution $\tilde{U} \in C([0, \infty), H)$ and $\tilde{U} \in C([0, \infty), D({A}(0))) \cap C^1([0, \infty), H)$, and consequently the solution of~\eqref{eq:Cauchy} is $U(t)=e^{\int_0^t \kappa(s) d s} \tilde{U}(t)$.
\end{proof}

The next proposition states that the energy \eqref{eq:En} is decreasing along the solutions of \eqref{eq:Cauchy}. The proof is straightforward and hence omitted.

\begin{proposition}\label{pr:Diss}
Suppose $\alpha$ and $\beta$ are real constants such that~\eqref{eq:CCond} holds. Then, for any mild solution of~\eqref{eq:Cauchy} the energy $E(t)$ defined by~\eqref{eq:En} is non-increasing and
\begin{equation}\label{eq:EnDiss3}
\frac{d}{dt}E(t)
 =
 \frac{1}{2} \begin{pmatrix}
\eta_x(L) \\ \eta_x(t-\tau(t),L)
\end{pmatrix}^{T}
\Phi_{\alpha,\beta}
\begin{pmatrix}
\eta_x(L) \\ \eta_x(t-\tau(t),L)
\end{pmatrix}.
\end{equation}
\end{proposition}

We end this section by giving \emph{a priori} estimates and the Kato smoothing effect which are essential to obtain the well-posedness of the system~\eqref{eq:KdV-KdV}. Here, we consider $(S_t(s))_{s\geq 0}$ to be the semigroup of contractions associated with the operator $A(t)$.
\begin{proposition}\label{pr:Kato}
	Let $\alpha$ and $\beta$ are real constant  such that~\eqref{eq:CCond} holds. Then, the map
	\begin{equation*}
		(\eta_0,\omega_0;z_0)\in{H} \mapsto (\eta,\omega;z) \in \mathcal{B} \times C(0,T; L^2(0,1))
	\end{equation*}
	is well-defined, continuous, and fulfills
	\begin{equation}\label{eq:Kato1}
		\lVert
		(\eta,\omega)
		\rVert_{X_0}^2
		+ \lvert \beta \rvert
		\lVert
		z
		\rVert_{L^2(0,1)}^2
		\leq
		\lVert
		(\eta_0,\omega_0)
		\rVert_{X_0}^2
		+ \lvert \beta \rvert
		\lVert
		z_0(-\tau(0)\cdot)
		\rVert_{L^2(0,1)}^2,
	\end{equation}
	Furthermore, for every $(\eta_0,\omega_0, z_0)\in {H}$, we have that
	\begin{equation}\label{eq:Katotr0}
		\lVert \eta_x(\cdot,L) \rVert_{L^2(0,T)}^2
		+\lVert z(\cdot,1) \rVert_{L^2(0,T)}^2
		\leq  \lVert(\eta_0,\omega_0)\rVert_{X_0}^2 +
		\lVert z_0(-\tau(0)\cdot)\rVert_{L^2(0,1)}^2.
	\end{equation}
	
	Moreover, the Kato smoothing effect  is verified
	\begin{equation}\label{eq:Kato2}
		\int_0^T\int_0^L
		\left(\eta_x^2+\omega_x^2\right)
		\,dx\,dt
		\leq
		C(L,T,\alpha)
		\left(
		\lVert (\eta_0,\omega_0) \rVert_{X_0}^2
		+
		\lVert z_0(-\tau(0)\cdot) \rVert_{L^2(0,1)}^2
		\right).
	\end{equation}
	Finally, for the initial data, we have the following estimates
	\begin{equation}\label{eq:Kato3}
		\begin{split}
		\lVert
		(\eta_0,\omega_0)
		\rVert_{X_0}^2
		\leq&
		\frac{1}{T}
		\lVert
		(\eta,\omega)
		\rVert_{L^2(0,T; X_0)}^2
		\\&+(2\alpha+\lvert\beta\rvert)
		\lVert
		\eta_x(\cdot,L)
		\rVert_{L^2(0,T)}^2
		+\lvert\beta\rvert
		\lVert
		z(\cdot,1)
		\rVert_{L^2(0,1)}^2
		\end{split}
	\end{equation}
	and
	\begin{equation}\label{eq:Kato4}
		\lVert z_0(-\tau(0)\cdot)\rVert_{L^2(0,1)}^2 \leq
		C_1(d,M)\left(\lVert z(T,\cdot )\rVert_{L^2(0,1)}
		+  \lVert z(\cdot, 1) \rVert_{L^2(0,T)}^2\right).
	\end{equation}
\end{proposition}
\begin{proof}
	From~\eqref{eq:EnDiss3} and using that $\Phi_{\alpha,\beta}$ is  a symmetric negative definite matrix we obtain that $E'(t) + \eta_x^2(t,L) + z^2(t,1) \leq 0$. Integrating in $[0,s]$, for $0\leq s \leq T$, we get
	\begin{equation}
		E(s) + \int_0^s \eta_x^2(t,L)\,dt + \int_0^s z^2(t,1)\,dt \leq E(0),
	\end{equation}
	and \eqref{eq:Kato1} is obtained. Taking $s=T$ and since $E(t)$ is a non-increasing function (see Proposition~\ref{pr:Diss}), the estimate~\eqref{eq:Katotr0} holds. Now, multiplying the first equation of the linearized system associated with \eqref{eq:KdV-KdV} by $x\omega$ and the second one by $x\eta$, adding the results, then integrating by parts in $(0,L)\times(0,T)$ and using~\eqref{eq:Katotr0}, we obtain
	\begin{equation}
		\begin{aligned}
			&\frac{3}{2}\int_0^T\int_0^L
			\eta_x^2+\omega_x^2
			\,dx\,dt
			%=&
			%\int_0^L x\eta_0\omega_0 -x\eta(T,x)\omega(T,x) \,dx 	\\&+\frac{1}{2}\int_0^T\int_0^L	\eta^2+\omega^2\,dx\,dt \\&+\frac{L}{2}\int_0^T
			%(-\alpha\eta_x(t,L)+\beta z(t,1))^2+\eta_x^2(t,L)\,dt \\
			\leq
			(L+T) \lVert (\eta_0,\omega_0) \rVert_{X_0}^2
			\\&+
			\left(\alpha^2+\frac{1}{2}\right) L
			\left(
			\lVert\eta_x(\cdot,L) \rVert_{L^2(0,T)}^2
			+
			\lVert z(\cdot,1)\rVert_{L^2(0,T)}^2
			\right) \\
			\leq& C(L,T,\alpha)\left( \lVert(\eta_0,\omega_0)\rVert_{X_0}^2
			+\lVert z_0(-\tau(0)\cdot) \rVert_{L^2(0,1)}^2\right),
		\end{aligned}
	\end{equation}
	where $C(L,T,\alpha) := \max\left\lbrace
	1, L+T, \left(\alpha^2+\frac{1}{2}\right) L
	\right\rbrace,$
	showing \eqref{eq:Kato2}. Secondly, we multiply the first equation of the linearized system associated with \eqref{eq:KdV-KdV} by $(T-t)\eta$, while the second one is multiplied by $(T-t)\omega$. Then,  adding the results yields
	\begin{equation*}
 \begin{split}
		\frac{T}{2}\lVert (\eta_0,\omega_0) \rVert_{X_0}^2
		\leq&
		\frac{1}{2}\lVert (\eta,\omega) \rVert_{L^2(0,T,X_0)}^2
		+T
		\left(
		\alpha+\frac{\lvert\beta\rvert}{2}
		\right)
		\int_0^T
		\eta_x^2(t,L)\,dt\\
		&+T\frac{\lvert\beta\rvert}{2}
		\int_0^T z^2(t,1)\,dt,
 \end{split}
	\end{equation*}
	where we have used Young's inequality, verifying  \eqref{eq:Kato3}. Finally, multiplying~\eqref{eq:tr}$_1$ by $z$ and integrating by parts in $(0,T)\times(0,1)$,
	\begin{equation*}
		\tau_0 \int_0^1 z_0^2(-\tau(0)\rho)\,d\rho \leq
		\int_0^T (1-\dot\tau(t))z^2(t,1)\,dt
		+\tau(T)\int_0^1 z^2(T,\rho)\,d\rho,
	\end{equation*}
	giving \eqref{eq:Kato4}.
\end{proof}

The next result ensures the existence of solutions to the KdV-KdV system with source terms.

\begin{theorem}\label{th:APsou}
	Suppose  that~\eqref{eq:CCond} and~\eqref{eq:TauCond} holds. Let $U_0 = (\eta_0, \omega_0 , z_0) \in H$ and the source terms $(f_1,f_2)\in L^1(0,T; X_0)$. Then there exists a unique solution $U = (\eta,\omega, z) \in C([0,T], H)$ to
	\begin{equation}\label{eq:KK+Sou}
		\begin{cases}
			\eta_t(t,x)+\omega_x(t,x)+\omega_{xxx}(t,x) =f_1, & t>0, x\in(0,L), \\
			\omega_t(t,x)+\eta_x(t,x)+\eta_{xxx}(t,x) = f_2, & t>0, x\in(0,L), % \\
			%\eta(t,0) = \eta(t,L) = \eta_x(t,0) = \omega(t,0) = \omega(t,L)= 0, & t>0,
		%	\omega_x(t,L) = -\alpha \eta_x(t,L)+\beta \eta_x(t-\tau(t),L), & t>0, \\
		%	\eta_x(t,L) = z_0(t), \left(\eta(0,x),\omega(0,x)\right) = \left(\eta_0(x),\omega_0(x)\right)& t\in(-h,0), x\in(0,L).
		\end{cases}
	\end{equation}
	with boundary condition as in \eqref{eq:KdV-KdV}.
	Moreover, for $T > 0$, the following estimates hold
	\begin{equation}\label{eq:APsou1}
\left\{	\begin{array}{l}
		\lVert (\eta,\omega; z) \rVert_{C([0,T], H)} \leq
		C (\lVert (\eta_0,\omega_0, z_0) \rVert_H + \lVert (f,g) \rVert_{L^1(0,T, X_0)} ),
\\[2mm]
		\lVert(\eta_x(\cdot,L), z(\cdot,1))\rVert_{[L^2(0,T)]^{2}}^2
		\leq C(\lVert{(\eta_0,\omega_0,z_0)}\rVert_{H}^2 + \lVert (f,g) \rVert_{L^1(0,T, X_0)}^2),
\\[2mm]
\lVert (\eta,\omega) \rVert_{L^2([0,T], X_1)} \leq
		C (\lVert (\eta_0,\omega_0, z_0) \rVert_H + \lVert (f,g) \rVert_{L^1(0,T, X_0)}),
		\end{array}
		\right.
	\end{equation}
	for some constant $C>0$.
\end{theorem}
\begin{proof}
Analogously to the proof of Proposition \ref{pr:Kato}, it suffices to use~\eqref{eq:EnDiss3} and take into account that $\Phi_{\alpha,\beta}$ is a symmetric negative definite matrix. This implies that there exists $C>0$ such that
$$E'(t) +
		\eta_x^2(t,L) + z^2(t,1)
		\leq C \left\langle (\eta,\omega),\ (f_1,f_2)\right\rangle_{X_0}.$$
Integrating the previous inequality on $[0,s]$ for $0\leq s\leq T$, we get
	\begin{equation}\label{eq:APaux1}
		E(s) + \int_0^s \eta_x^2(t,L)\,dt + \int_0^s z^2(t,1)\,dt
		\leq
		C \left(
		\int_0^s \left\langle (\eta,\omega), (f_1,f_2)\right\rangle_{X_0} + E(0)
		\right).
	\end{equation}
	From Cauchy-Schwarz inequality, it follows that
	\begin{equation*}
	\begin{split}
		\left\lVert (\eta(s,\cdot), \omega(s,\cdot) ; z(s,\cdot) \right\rVert_{H}^2
		\leq&
		C \left(
		\lVert (\eta_0,\omega_0; z_0) \rVert_{H}^2\right.\\&
		\left.+
		\lVert (f_1,f_2) \rvert_{L^1(0,T; X_0)}\lVert (\eta,\omega) \rVert_{C([0,T],X_0)}
		\right),
		\end{split}
	\end{equation*}
and consequently, taking the $\sup$-norm for $s\in[0,T]$ and applying  Young's inequality, the estimate $\eqref{eq:APsou1}_{1}$ is obtained. Additionally, if we consider $s=T$ in~\eqref{eq:APaux1}, the estimate for the traces~$\eqref{eq:APsou1}_{2}$ is guaranteed.
	Finally, by using the same Morawetz multipliers  as in Proposition~\ref{pr:Kato},  we have
	\begin{equation*}
		\begin{aligned}
			\int_0^T\int_0^L x\left(f_1\omega(t,x) + f_2\eta(t,x) \right)dxdt
			%& \leq  L\int_0^T  \left\langle (\omega,\eta), (f_1,f_2) \right\rangle_{X_0}\,dt\\
			%\leq& L\int_0^T \lVert(\eta,\omega)\rVert_{X_0}\lVert(f_1,f_2)\rVert_{X_0}\,dt \\
			%&\leq L\lVert(\eta,\omega)\rVert_{	C([0,T],X_0)}\int_0^T \lVert(f_1,f_2)\rVert_{X_0}\,dt \\
			\leq &  L\lVert(\eta,\omega, z)\rVert_{C([0,T],H)} \lVert(f_1,f_2)\rVert_{L^1(0,T,X_0)}\\
			%\leq& CL \lVert(\eta_0,\omega_0, z_0)\rVert_{C([0,T],H)} \lVert(f_1,f_2)\rVert_{L^1(0,T,X_0)} \\
			%&+CL \lVert(f_1,f_2)\rVert_{L^1(0,T, X_0)}^2 \\
			\leq& C
			\left(
			\lVert(\eta_0,\omega_0; z_0) \rVert_{H}^2 + \lVert (f_1,f_2) \rVert_{L^1(0,T, X_0)}^2
			\right),
		\end{aligned}
	\end{equation*}
	proving $\eqref{eq:APsou1}_{3}$.
\end{proof}

\subsection{Nonlinear problem}
Using the theory of local well-posedness of nonlinear systems in~\cite{Kato1975}, it amounts to proving that the map $\Gamma \colon \mathcal{B} \to \mathcal{B}$ has a unique fixed-point in some closed ball $B(0, R)\subset \mathcal{B}$ where $\Gamma(\tilde\eta,\tilde\omega) = (\eta,\omega)$ and $(\eta,\omega)$ are the solution of the system \eqref{eq:KdV-KdV}. The next result ensures that the nonlinear terms can be considered as a source term of the linear equation~\eqref{eq:KK+Sou}. The proof can be found in~\cite{Capistrano2019}.

\begin{proposition}\label{pr:RC_nl}
Let $(\eta,\omega)\in L^2(0,T, [H^1(0,L)]^2)$, so $(\eta\omega)_x$, $\omega\omega_x\in L^1(0,T, X_0)$ and
$
(\eta,\omega) \in\mathcal{B} \mapsto ((\eta\omega)_x, (\omega\omega_x))\in L^1(0,T, X_0)
$
is continuous. In addition, the following estimate holds,
\begin{equation*}
\begin{split}
\int_0^T \left\lVert ((\eta_1\omega_1)_x-(\eta_2\omega_2)_x, \omega_1\omega_{1,x} - \omega_2\omega_{2,x}) \right\rVert_{X_0}dt
\leq&
KT^{\frac{1}{4}}\left(\lVert (\eta_1,\omega_1) \rVert_{\mathcal{B}} +\lVert (\eta_2,\omega_2)\rVert_{\mathcal{B}}\right)\\
&\times\lVert (\eta_1-\eta_2, \omega_1-\omega_2 ) \rVert_{\mathcal{B}}\,
\end{split}
\end{equation*}
for a constant $K>0$.
\end{proposition}

Finally, we are in a position to present the existence of local solutions to \eqref{eq:KdV-KdV}.

\begin{theorem}\label{th:NLinSol}
Let $L, T> 0$ and consider $\alpha$ and $\beta$  real constants such that~\eqref{eq:CCond} is satisfied. For each initial data $(\eta_0, \omega_0; z_0) \in  H$ sufficiently small,  $\Gamma\colon \mathcal{B} \to\mathcal{B}$ defined by $\Gamma(\tilde\eta,\tilde\omega) = (\eta,\omega)$ is a contraction. Moreover, there exists a unique solution $(\eta, \omega) \in B(0, R) \subset \mathcal{B}$   of the Boussinesq KdV-KdV nonlinear system \eqref{eq:KdV-KdV}.
\end{theorem}
\begin{proof}
It follows from Theorem~\ref{th:APsou} that the map $\Gamma$ is well defined. Using Proposition~\ref{pr:RC_nl} and the \emph{a priori estimates} we obtain
\begin{equation*}
\lVert \Gamma(\tilde\eta,\tilde\omega)\rVert_{\mathcal{B}}
= \lVert (\eta,\omega) \rVert_{\mathcal{B}}
\leq C\left( \lVert(\eta_0,\omega_0, z_0)\rVert_{H} + \lVert (\tilde\eta,\tilde\omega)\rVert_{\mathcal{B}}^2 \right)
\end{equation*}
and
\begin{equation*}
\left\lVert \Gamma(\tilde\eta_1,\tilde\omega_1) - \Gamma(\tilde\eta_2,\tilde\omega_2) \right\rVert_{\mathcal{B}}
\leq
K T^{\frac{1}{4}} \left( \lVert (\tilde\eta_1,\tilde\omega_1) \rVert_{\mathcal{B}} + \lVert (\tilde\eta_2,\tilde\omega_2)\rVert_{\mathcal{B}}\right) \lVert (\tilde\eta_1-\tilde\eta_2, \tilde\omega_1-\tilde\omega_2 ) \rVert_{\mathcal{B}}.
\end{equation*}
Now, we restrict $\Gamma$ to the closed ball $\lbrace(\tilde\eta,\tilde\omega)\in\mathcal{B}: \lVert (\tilde\eta,\tilde\omega)\rVert_{\mathcal{B}}\leq R\rbrace$, with $R > 0$ to be determined later. Then,
$
\lVert \Gamma(\tilde\eta,\tilde\omega)\rVert_{\mathcal{B}}
\leq C\left( \lVert(\eta_0,\omega_0, z_0)\rVert_{H} +R^2 \right)$
and
$$
\left\lVert \Gamma(\tilde\eta_1,\tilde\omega_1) - \Gamma(\tilde\eta_2,\tilde\omega_2) \right\rVert_{\mathcal{B}}
\leq
2RK T^{\frac{1}{4}} \lVert (\tilde\eta_1-\tilde\eta_2, \tilde\omega_1-\tilde\omega_2 ) \rVert_{\mathcal{B}}.
$$
Next, we pick $R = 2C\lVert(\eta_0,\omega_0, z_0)\rVert_{H} $ and $T>0$ such that $2KT^{\frac{1}{4}} R < 1$, with $C < 2KT^{\frac{1}{4}}$. This leads to claim that $$\lVert \Gamma(\tilde\eta,\tilde\omega)\rVert_{\mathcal{B}}\leq R$$ and  $$\left\lVert \Gamma(\tilde\eta_1,\tilde\omega_1) - \Gamma(\tilde\eta_2,\tilde\omega_2) \right\rVert_{\mathcal{B}} < C_1  \lVert (\tilde\eta_1-\tilde\eta_2, \tilde\omega_1-\tilde\omega_2 ) \rVert_{\mathcal{B}},$$ with $C_1<1$. Lastly, the result is an immediate consequence of the Banach fixed point theorem.
\end{proof}

\begin{remark}We point out that the solutions of the system \eqref{eq:KdV-KdV} obtained in Theorem \ref{th:NLinSol} are only local.  Due to a lack of a priori $L^2$-estimate, the issue of the global existence of solutions is difficult to address in the energy space for the nonlinear system with a delay term.
\end{remark}

\section{Linear stabilization result}\label{sec3}
Since the $L^2$ a priori estimate is valid for the linear system, the solutions of the linearized system associated with \eqref{eq:KdV-KdV} are globally well-posed. Thereby, we are ready to prove the main result of this work.

\subsection{Proof of Theorem~\ref{th:Lyapunov0}} Consider the following Lyapunov functional $$V(t) = E(t) + \mu_1V_1(t) + \mu_2 V_2(t),$$ where $\mu_1,\mu_2 \in \mathbb{R}^+$ will be chosen later. Here, $E(t)$ is the total energy given by \eqref{eq:En}, while
$$
V_1(t) =\frac{1}{2} \int_0^L x\eta(t,x)\omega(t,x)\,dx $$
and $$ V_2(t) = \frac{\lvert\beta\rvert}{2} \tau(t) \int_0^1 (1-\rho) \eta_x^2(t-\tau(t)\rho ,L)\,d\rho.$$
Observe that,
\begin{equation*}
	(1-\max\lbrace \mu_1 L, \mu_2\rbrace)E(t) \leq V(t) = E(t) + \mu_1 V_1(t) + \mu_2 V_2(t) \leq (1+\max\lbrace  2\mu_1L, \mu_2 \rbrace) E(t).
\end{equation*} 	
The Young's inequality yields that
\begin{equation}
	\left\lvert\mu_1\int_0^L x\eta\omega \,dx \right\rvert
	\leq \mu_1 L\int_0^L \left\lvert\eta\omega\right\rvert \,dx
	\leq \frac{\mu_1 L}{2} \int_0^L\left( \eta^2 + \omega^2\right)dx.
\end{equation}
Moreover,
\begin{equation}
	\begin{aligned}
		\left\lvert  \mu_1\int_0^L x\eta\omega \,dx\right. &   \left.  +\ \mu_2\cdot \frac{\lvert\beta\rvert}{2} \tau(t) \int_0^1 (1-\rho) \eta_x^2(t-\tau(t)\rho ,L)\,d\rho \right\rvert\\
		\leq \ &
		\left\lvert\mu_1\int_0^L x\eta\omega \,dx \right\rvert + \left\lvert \mu_2\cdot \frac{\lvert\beta\rvert}{2} \tau(t) \int_0^1 (1-\rho) \eta_x^2(t-\tau(t)\rho ,L)\,d\rho \right\rvert \\
		\leq \ &
		\frac{\mu_1 L}{2} \int_0^L \left(\eta^2 + \omega^2\right)dx +  \mu_2\cdot \frac{\lvert\beta\rvert}{2} \tau(t) \int_0^1 \eta_x^2(t-\tau(t)\rho ,L)\,d\rho \\
		\leq \ &\max\lbrace \mu_1 L,\mu_2\rbrace
		\left(
		\frac{1}{2} \int_0^L  \left(\eta^2 + \omega^2\right)dx + \frac{\lvert \beta\rvert }{2}\tau(t) \int_0^1 \eta_x^2(t-\tau(t)\rho , L)\,d\rho.			\right)\\
		= \ & \max\lbrace \mu_1 L,\mu_2\rbrace E(t),
	\end{aligned}
\end{equation}
and, consequently,
\begin{equation}\label{eq:EquivEV}
	(1-\max\lbrace \mu_1 L, \mu_2\rbrace) E(t) \leq V(t) \leq (1+\max\lbrace\mu_1 L,\mu_2\rbrace)E(t),
\end{equation} 	
since $\mu_1 L<1$ by hypothesis.

To obtain the derivative of $V_1$, we have
\begin{equation*}
V_1'(t)=\dfrac{d}{dt}\left(\displaystyle\int_{0}^{L}x\eta \omega dx\right)=\displaystyle\int_{0}^{L} x\eta_t\omega dx+\int_{0}^{L}x\eta \omega_t dx=I_1+I_2.
\end{equation*}
Let us analyze each term. For $I_1$, using the boundary condition, we get that
\begin{align*}
	\int_{0}^{L}x\eta_t\omega dx
	%=& \int_{0}^{L}x\omega\left(-\omega_x-\omega_{xxx}\right)dx\\
	=& -\int_{0}^{L}x\omega_x\omega dx-\int_{0}^{L}x\omega_{xxx}\omega dx\\
	%=& -\int_{0}^{L}x\frac{1}{2}(\omega^2)_xdx+\int_{0}^{L}\omega_{xx}(x\omega)_xdx\\
	%=&\frac{1}{2}\int_{0}^{L}\omega^2dx+\int_{0}^{L}\omega_{xx}\left(\omega+x\omega_x\right)dx\\
	=& \frac{1}{2} \int_{0}^{L}\omega^2dx-\int_{0}^{L}\omega_{x}^{2}dx+\int_{0}^{L}\frac{x}{2}(\omega_x^2)_xdx\\
	=& \frac{1}{2} \int_{0}^{L}\omega^2dx-\int_{0}^{L}\omega_{x}^{2}dx+ \left(\frac{x}{2}(\omega_{x}^{2})\right)\big|_{0}^{L}-\frac{1}{2}\int_{0}^{L}\omega_{x}^{2}dx\\
	=& \frac{1}{2}\int_{0}^{L} \omega^2 dx-\int_{0}^{L} \omega_{x}^{2}dx+\frac{L}{2}(-\alpha \eta_x(L)+\beta \eta_x(t-\tau(t),L))^2-\frac{1}{2}\int_{0}^{L}\omega_{x}^{2}dx.
\end{align*}
Therefore,
\begin{equation}\label{I_1}
	\begin{split}
		\int_{0}^{L}x\eta_t\omega dx=& \frac{1}{2} \int_{0}^{L}\omega^2dx-\frac{3}{2}\int_{0}^{L}\omega_{x}^{2}dx\\
		&+\frac{L}{2}(-\alpha \eta_x(L)+\beta \eta_x(t-\tau(t),L))^2.
	\end{split}
\end{equation}
For $I_2$, thanks to the boundary condition, we have that
\begin{equation}\label{I_2}
\begin{split}
	\int_{0}^{L}x \eta \omega_t dx
	%=& \int_{0}^{L} x \eta \left(-\eta_x-\eta_{xxx}\right)dx\\
	=& -\int_{0}^{L}x\eta \eta_x dx-\int_{0}^{L} x \eta \eta_{xxx} dx\\
	%=& -\int_{0}^{L} \frac{x}{2}(\eta^2)_xdx+\int_{0}^{L}\eta_{xx}\left(x\eta\right)_xdx\\
	%=& \frac{1}{2}\int_{0}^{L}\eta^2dx+\int_{0}^{L} \eta_{xx}\left(\eta+x\eta_x\right)dx\\
	%=& \frac{1}{2}\int_{0}^{L}\eta^2 dx+\int_{0}^{L}\eta_{xx}\eta dx+\int_{0}^{L}\frac{x}{2}\left(\eta_{x}^{2}\right)_xdx\\
	=& \frac{1}{2}\int_{0}^{L}\eta^2dx-\int_{0}^{L}\eta_{x}^{2}dx+\left(\frac{x}{2} \eta_{x}^{2}\right)\big|_{0}^{L}-\frac{1}{2}\int_{0}^{L}\eta_{x}^{2}dx\\
	=&\frac{1}{2}\int_{0}^{L}\eta^2dx-\int_{0}^{L}\eta_{x}^{2}dx-\frac{1}{2}\int_{0}^{L}\eta_{x}^{2}dx+\frac{L}{2} \eta_{x}^{2}(L)\\
	=&\frac{1}{2}\int_{0}^{L}\eta^2dx-\frac{3}{2}\int_{0}^{L}\eta_{x}^{2}dx-\int_{0}^{L}x\eta \omega \omega_{x} dx+\frac{L}{2} \eta_{x}^{2}(L).
\end{split}
\end{equation}
Adding the identities \eqref{I_1} and \eqref{I_2} we obtain the following identity
\begin{align*}
	V_1'(t)=& \frac{1}{2} \int_{0}^{L}\omega^2dx-\frac{3}{2}\int_{0}^{L}\omega_{x}^{2}dx+\frac{L}{2}(-\alpha \eta_x(L)+\beta \eta_x(t-\tau(t),L))^2\\
	&+\frac{1}{2}\int_{0}^{L}\eta^2dx-\frac{3}{2}\int_{0}^{L}\eta_{x}^{2}dx+\frac{L}{2} \eta_{x}^{2}(L).
\end{align*}
Hence,
\begin{align*}
	V_1'(t)=&\frac{L}{2} \begin{pmatrix}
		\eta_x(t,L) \\ \eta_x(t-\tau(t),L)
	\end{pmatrix}^T
	\begin{pmatrix}
		\alpha^2+1 & -\alpha\beta \\
		-\alpha\beta & \beta^2
	\end{pmatrix}
	\begin{pmatrix}
		\eta_x(t,L) \\ \eta_x(t-\tau(t),L)
	\end{pmatrix}\\
	&+\frac{1}{2} \int_{0}^{L}\left(\omega^2+\eta^2\right)dx-\frac{3}{2}\int_{0}^{L}\left(\omega_{x}^{2}+\eta_{x}^{2}\right)dx.
\end{align*}

Let
\begin{equation*}
	V_2(t) = \frac{\lvert\beta\rvert}{2} \tau(t) \int_0^1 (1-\rho) \eta_x^2(t-\tau(t)\rho ,L) d\rho.	
\end{equation*}
Remembering that $$-\tau(t)\partial_t\eta_x(t-\tau(t)\rho,L) = (1-\dot\tau(t)\rho)\partial_{\rho}\eta_x(t-\tau(t)\rho,L)$$ we have, by integration by parts, that
\begin{equation*}
	\begin{aligned}
		V_2'(t) =\ & \frac{\lvert\beta\rvert}{2} \dot\tau(t)\int_0^1 (1-\rho) \eta_x^2(t-\tau(t)\rho,L) d\rho \\
		&+\lvert\beta\rvert \tau(t)\int_0^1 (1-\rho)\eta_x(t-\tau(t-\tau(t)\rho,L)\partial_t\eta_x(t-\tau(t)\rho,L) d\rho \\
		=\ &\frac{\lvert\beta\rvert}{2} \dot\tau(t)\int_0^1 (1-\rho) \eta_x^2(t-\tau(t)\rho,L)d\rho \\
		&+\lvert\beta\rvert\int_0^1 (\rho-1) (1-\dot\tau(t)\rho) \eta_x(t-\tau(t)\rho,L)\partial_\rho \eta_x(t-\tau(t)\rho,L) d\rho \\
		=\ &\frac{\lvert\beta\rvert}{2}\dot\tau(t)\int_0^1 (1-\rho) \eta_x^2(t-\tau(t)\rho,L) d\rho
		+\frac{\lvert\beta\rvert}{2} \int_0^1 (\rho-1)(1-\dot\tau(t)\rho)\left(\eta_x^2(t-\tau(t)\rho,L)\right)_\rho d\rho \\
		=\ &\frac{\lvert\beta\rvert}{2}\dot\tau(t)\int_0^1 (1-\rho) \eta_x^2(t-\tau(t)\rho,L) d\rho +\frac{\lvert\beta\rvert}{2} \int_0^1 \left[(1-\rho)(1-\dot\tau(t)\rho)\right]_{\rho} \eta_x^2(t-\tau(t)\rho,L) d\rho \\
		& +\frac{\lvert\beta\rvert}{2} \left[ (\rho-1)(1-\dot\tau(t)\rho)\eta_x^2(t-\tau(t)\rho,L)\right]_{\rho=0}^{\rho=1} \\
		=\ &-\frac{\lvert\beta\rvert}{2}\int_0^1 (1-\dot\tau(t)\rho) \eta_x^2(t-\tau(t)\rho,L) d\rho + \frac{\lvert\beta\rvert}{2}\eta_x^2(t,L),
	\end{aligned}
\end{equation*}
that is,
\begin{equation}\label{V_2'}
	V_2'(t)  = -\frac{\lvert\beta\rvert}{2}\int_0^1(1-\dot\tau(t)\rho)\eta_x^2(t-\tau(t)\rho,L) d\rho + \frac{\lvert\beta\rvert}{2}\eta_x^2(t,L).
\end{equation}
Since the energy of our problem is given by
\begin{equation*}
	E(t) = \frac{1}{2} \int_0^L\left(\eta^2+\omega^2\right) dx
	+ \frac{\lvert \beta\rvert }{2}\tau(t) \int_0^1 \eta_x^2(t-\tau(t)\rho , L) d\rho,
\end{equation*}
yields that
\begin{equation*}
	E'(t)
	=\frac{1}{2} \begin{pmatrix}
		\eta_x(t,L) \\ \eta_x(t-\tau(t),L)
	\end{pmatrix}^T
	\Phi_{\alpha,\beta}\begin{pmatrix}
		\eta_x(t,L) \\ \eta_x(t-\tau(t),L)
	\end{pmatrix},
\end{equation*}
with
\begin{equation*}
	\Phi_{\alpha,\beta} = \begin{pmatrix}
		-2\alpha + \lvert\beta\rvert & \beta\\
		\beta & \lvert\beta\rvert (d -1 )
	\end{pmatrix}.
\end{equation*}
Let
\begin{equation*}
	V(t) = E(t) + \mu_1 V_1(t) + \mu_2 V_2(t).
\end{equation*}
Then,
\begin{align*}
	V'(t)+\lambda V(t)=&E'(t)+\mu_1 V_1'(t)+\mu_2 V_2'(t)+\lambda E(t)+\lambda \mu_1 V_1(t)+\lambda \mu_2 V_2(t)\\
	=& \frac{1}{2} \begin{pmatrix}
			\eta_x(t,L) \\ \eta_x(t-\tau(t),L)
		\end{pmatrix}^T
		\Phi_{\alpha,\beta}\begin{pmatrix}
			\eta_x(t,L) \\ \eta_x(t-\tau(t),L)
	\end{pmatrix}
	\\
	&+\frac{\mu_1L}{2} \begin{pmatrix}
			\eta_x(t,L) \\ \eta_x(t-\tau(t),L)
		\end{pmatrix}^T
		\begin{pmatrix}
			\alpha^2+1 & -\alpha\beta \\
			-\alpha\beta & \beta^2
		\end{pmatrix}
		\begin{pmatrix}
			\eta_x(t,L) \\ \eta_x(t-\tau(t),L)
	\end{pmatrix}\\
	&+\frac{\mu_1}{2} \int_{0}^{L}\left(\omega^2+\eta^2\right)dx-\frac{3\mu_1}{2}\int_{0}^{L}\left(\omega_{x}^{2}+\eta_{x}^{2}\right)dx\\
	&-\mu_2\frac{\lvert\beta\rvert}{2}\int_0^1(1-\dot\tau(t)\rho)\eta_x^2(t-\tau(t)\rho,L) d\rho +\frac{\mu_2\lvert\beta\rvert}{2}\eta_x^2(t,L)\\
	&+\frac{\lambda}{2} \int_0^L\left(\eta^2+\omega^2\right)dx
	+ \frac{\lambda \lvert \beta\rvert }{2}\tau(t) \int_0^1 \eta_x^2(t-\tau(t)\rho , L) d\rho\\
	& +\mu_1\lambda\int_{0}^{L} x\eta \omega dx+\frac{\mu_2\lvert\beta\rvert{ \lambda}}{2} \tau(t) \int_0^1 (1-\rho) \eta_x^2(t-\tau(t)\rho ,L) d\rho.
\end{align*}
Therefore,
\begin{align*}
	V'(t)+\lambda V(t)=& \frac{1}{2} \left\langle \Psi_{\mu_1,\mu_2} (\eta_x(t,L), \eta_x(t-\tau(t),L)), (\eta_x(t,L), \eta_x(t-\tau(t),L)) \right\rangle \\
	&+\frac{\mu_1}{2} \int_{0}^{L}\left(\omega^2+\eta^2\right)dx-\frac{3\mu_1}{2}\int_{0}^{L}\left(\omega_{x}^{2}+\eta_{x}^{2}\right)dx\\&+\frac{\lambda}{2} \int_0^L\left(\eta^2+\omega^2\right) dx+\mu_1\lambda\int_{0}^{L} x\eta \omega dx\\
	&-\mu_2\frac{\lvert\beta\rvert}{2}\int_0^1(1-\dot\tau(t)\rho)\eta_x^2(t-\tau(t)\rho,L) d\rho	+ \frac{\lambda \lvert \beta\rvert }{2}\tau(t) \int_0^1 \eta_x^2(t-\tau(t)\rho , L) d\rho\\
	&+\frac{\mu_2\lvert\beta\rvert{ \lambda}}{2} \tau(t) \int_0^1 (1-\rho) \eta_x^2(t-\tau(t)\rho ,L) d\rho\\
	=& M+S_1+S_2.
\end{align*}
Here, the terms $\Psi_{\mu_1,\mu_2}$, $M$, $S_1$ and $S_2$ are given by
\begin{equation}\label{psi}
	\Psi_{\mu_1,\mu_2} =
	\Phi_{\alpha,\beta}
	+ L\mu_1 \begin{pmatrix}
		\alpha^2+1 & -\alpha\beta \\
		- \alpha\beta & \beta^2
	\end{pmatrix}
	+ \lvert\beta\rvert\mu_2 \begin{pmatrix}
		1 & 0 \\ 0 & 0
	\end{pmatrix},
\end{equation}

\begin{equation*}
	\begin{aligned}
		M=&	\frac{1}{2} \left\langle \Psi_{\mu_1,\mu_2} (\eta_x(t,L), \eta_x(t-\tau(t),L)), (\eta_x(t,L), \eta_x(t-\tau(t),L)) \right\rangle,\\
		S_1=& \frac{\mu_1}{2} \int_{0}^{L}\left(\omega^2+\eta^2\right)dx-\frac{3\mu_1}{2}\int_{0}^{L}\left(\omega_{x}^{2}+\eta_{x}^{2}\right)dx+\frac{\lambda}{2} \int_0^L\left(\eta^2+\omega^2\right) dx+\mu_1\lambda\int_{0}^{L} x\eta \omega dx,\\
			\end{aligned}
			\end{equation*}
			and
			\begin{equation*}
				\begin{aligned}
		S_2=&-\mu_2\frac{\lvert\beta\rvert}{2}\int_0^1(1-\dot\tau(t)\rho)\eta_x^2(t-\tau(t)\rho,L) d\rho + \frac{\lambda \lvert \beta\rvert }{2}\tau(t) \int_0^1 \eta_x^2(t-\tau(t)\rho , L) d\rho\\
		& +\frac{\mu_2\lvert\beta\rvert{ \lambda}}{2} \tau(t) \int_0^1 (1-\rho) \eta_x^2(t-\tau(t)\rho ,L) d\rho,
	\end{aligned}
\end{equation*}
respectively.

Now we need to prove that $V'(t)+\lambda V(t)\leq0$. To do that, let us analyze each term above.

\vspace{0.2cm}

\noindent\textbf{Estimate for $M$:} From the properties of $\Phi_{\alpha,\beta}$ and the continuity of the trace and determinant functions, we can ensure that % for $\mu_1,\mu_2\in(0,1)$ sufficiently small
$\Psi_{\mu_1,\mu_2}$ is negative definite. Thus,
$$
	M\leq0.
$$

\vspace{0.2cm}

\noindent\textbf{Estimate for $S_1$:} Observe that using Poincar\'e inequality, we get that
\begin{equation*}
	\begin{aligned}
		S_1
		%\leq& \frac{\mu_1}{2}\int_{0}^{L}\left( \omega^2+\eta^2\right) dx+\frac{\lambda}{2}\int_{0}^{L}\left(\eta^2+\omega^2\right)dx+\frac{\mu_1 \lambda L}{2}\int_{0}^{L} \left(\eta^2+\omega^2\right)dx-\frac{3\mu_1}{2}\int_{0}^{L}\left(\omega_{x}^{2}+\eta_{x}^{2}\right)dx\\
		\leq & \frac{1}{2}\left(\lambda(1+\mu_1L)+\mu_1 \right)\int_{0}^{L}\left( \omega^2+\eta^2 \right)dx -\frac{3\mu_1}{2}\int_{0}^{L}\left(\omega_{x}^{2}+\eta_{x}^{2}\right)dx\\
		\leq& \left[\frac{L^2}{2\pi^2}\left(\lambda(1+\mu_1L)+\mu_1 \right)-\frac{3\mu_1}{2}\right]\int_{0}^{L}\left(\omega_{x}^{2}+\eta_{x}^{2}\right)dx.
	\end{aligned}
\end{equation*}
Thus,
\begin{equation*}
	S_1<0,
\end{equation*}
if
\begin{equation*}
	\lambda < \frac{\mu_1(3\pi^2-L^2)}{L^2(1+\mu_1)}.
\end{equation*}

\vspace{0.2cm}

\noindent\textbf{Estimate for $S_2$:} Note that
\begin{equation*}
	\begin{aligned}
		S_2\leq& -\frac{\mu_2\lvert \beta \rvert }{2}(1-d)\int_0^1\eta_x^2(t-\tau(t)\rho,L) d\rho+\frac{\lambda \lvert \beta\rvert  M}{2}\int_0^1 \eta_x^2(t-\tau(t)\rho , L) d\rho\\
		&+\frac{\lambda\mu_2\lvert\beta\rvert M}{2} \int_0^1 \eta_x^2(t-\tau(t)\rho ,L) d\rho\\
		\leq& \frac{\lvert \beta \rvert}{2} \left(\lambda  M+\lambda\mu_2 M -\mu_2(1-d)\right)\int_0^1\eta_x^2(t-\tau(t)\rho,L) d\rho.
	\end{aligned}
\end{equation*}
Then, choosing
\begin{equation*}
	\lambda < \frac{\mu_2(1-d)}{M(1+\mu_2)}
\end{equation*}
we have that
$$
	\frac{\lvert \beta \rvert}{2} \left(\lambda  M+\lambda\mu_2 M -\mu_2(1-d)\right)<0.
$$

\color{black}
Therefore, for $\zeta>0$ and $\lambda> 0$ fulfilling \eqref{eq:r} and \eqref{eq:lambda}, respectively, we have
$$
\frac{d}{dt} V(t) + \lambda V(t) \leq 0 \iff E(t) \leq  \zeta E(0)e^{-\lambda t}, \quad \forall t\geq 0,
$$
since $V(t)$ satisfies \eqref{eq:EquivEV}. This achieves the proof of the theorem.\qed

\subsection{Optimization of the decay rate}
We can optimize the value of $\lambda$ in Theorem \ref{th:Lyapunov0} to obtain the best decay rate for the linear system associated with \eqref{eq:KdV-KdV} in the following way:
\begin{proposition}\label{llllll} Choosing the constant $\mu_1$  as follows
\begin{equation}\label{eq:muOPT}
	\mu_1 \in \left[0,\frac{(2\alpha-\lvert\beta\rvert)(1-d)-\lvert\beta\rvert}{L(1-d)(1+\alpha^2)} \right),
\end{equation}
we claim that $\lambda$ has the largest possible value.
\end{proposition}
\begin{proof}
Define the functions $f$ and $g$ $\colon \left[
	0, \frac{(2\alpha-\lvert\beta\rvert)(1-d)-\lvert\beta\rvert}{L(1-d)(1+\alpha^2)}
	\right]\to \mathbb{R}$
by
%\begin{equation*}
$$		f(\mu_1) = \frac{\mu_1 \left(3\pi^2-L^2\right)}{L^2(1+\mu_1L)},$$
%\end{equation*}
%\begin{equation*}
and
$$	g(\mu_1) = \frac{(2\alpha-|\beta|)(1-d)-|\beta|-L(1-d)(1+\alpha^2)\mu_1}{M(2\alpha(1-d)-|\beta|-L(1-d)(1+\alpha^2)\mu_1)}(1-d).$$
%\end{equation*}
Then, let $\lambda(\mu_1) = \min\lbrace f(\mu_1), g(\mu_1)\rbrace$ we have the following claims.

\begin{claim}\label{RR1} The function $f$ is increasing in the interval $\left[0,\frac{(2\alpha-\lvert\beta\rvert)(1-d)-\lvert\beta\rvert}{L(1-d)(1+\alpha^2)}\right)$ while the function $g$ is decreasing in the same interval.
\end{claim}

In fact, note that if
$$
	f(\mu_1)=\frac{\left(3 \pi^2-L^2\right)}{ L^3}\left(1-\frac{1}{1+ \mu_1 L}\right)\implies f^{\prime}(\mu_1)=\frac{\left(3 \pi^2-L^2\right)}{L^2(1+\mu_1 L)^2}>0.$$ In particular, $f^{\prime}(\mu_1)>0$ for $\mu_1 \in\left[0,\frac{(2\alpha-\lvert\beta\rvert)(1-d)-\lvert\beta\rvert}{L(1-d)(1+\alpha^2)}\right)$. Analogously, as
%\begin{equation*}
$$g(\mu_1)=\frac{1-d}{M}-\frac{|\beta|(1-d)^2}{ML(1-d)(1+\alpha^2)}\left(\frac{1}{\frac{2\alpha(1-d)-|\beta|}{L(1-d)(1+\alpha^2)}-\mu_1}\right),$$ so
%\end{equation*}
$$
	g^{\prime}(\mu_1)=-\frac{|\beta|(1-d)^2}{ML(1-d)(1+\alpha^2)}
	\left[
	\frac{1}{
		\left(
		\frac{2\alpha(1-d)-|\beta|}{L(1-d)(1+\alpha^2)}-\mu_1
		\right)^2
	}
	\right]<0,
$$
since $$\mu_1<\frac{(2\alpha-\lvert\beta\rvert)(1-d)-\lvert\beta\rvert}{L(1-d)(1+\alpha^2)}<\frac{2\alpha(1-d)-\lvert\beta\rvert}{L(1-d)(1+\alpha^2)},$$ showing the claim \ref{RR1}.

\begin{claim}\label{RR2}
There exists only one point satisfying~\eqref{eq:muOPT} such that $f(\mu_1)=g(\mu_1)$.
\end{claim}

Indeed, to show the existence of this point, it is sufficient to note that $f(0)=0,$
\begin{equation*}
\begin{split}
		&f\left(\frac{(2\alpha-\lvert\beta\rvert)(1-d)-\lvert\beta\rvert}{L(1-d)(1+\alpha^2)}\right)\\
		&=\frac{\left(3 \pi^2-L^2\right)}{2 L^3}\left(1-\frac{(1-d)(1+\alpha^2)}{(1-d)(1+\alpha^2)+(2\alpha-|\beta|)(1-d)-|\beta|}\right)>0
\end{split}
	\end{equation*}
	and
\begin{equation*}
g(0)=\frac{1-d}{M}\left(1-\frac{\lvert\beta\rvert(1-d)}{2\alpha(1-d)-|\beta|}\right)>0, \quad g\left(\frac{(2\alpha-\lvert\beta\rvert)(1-d)-\lvert\beta\rvert}{L(1-d)(1+\alpha^2)}\right)=0.
\end{equation*}
The uniqueness follows from the fact that $f$ is increasing while $g$ is decreasing in this interval, and claim \ref{RR2} holds.

\vspace{0.2cm}

Finally, taking into account the claims \ref{RR1} and \ref{RR2}, the maximum value of the function $\lambda$ must be reached at the point $\mu_1$ satisfying~\eqref{eq:muOPT}, where $f(\mu_1)=g(\mu_1)$, and the Proposition \ref{llllll} is achieved.
\end{proof}

We can illustrate, in Figure \ref{fig1} below, the situation of the previous proposition taking, for instance, $L=5$, $d=\frac{1}{2}$, $\alpha=1$, $\beta=\frac{1}{2}$ and $M=3$, when $\lambda(\mu_1)=\min\{f(\mu_1),g(\mu_1)\}$:

\begin{figure}[H]
	\centering
	\includegraphics[scale=0.45]{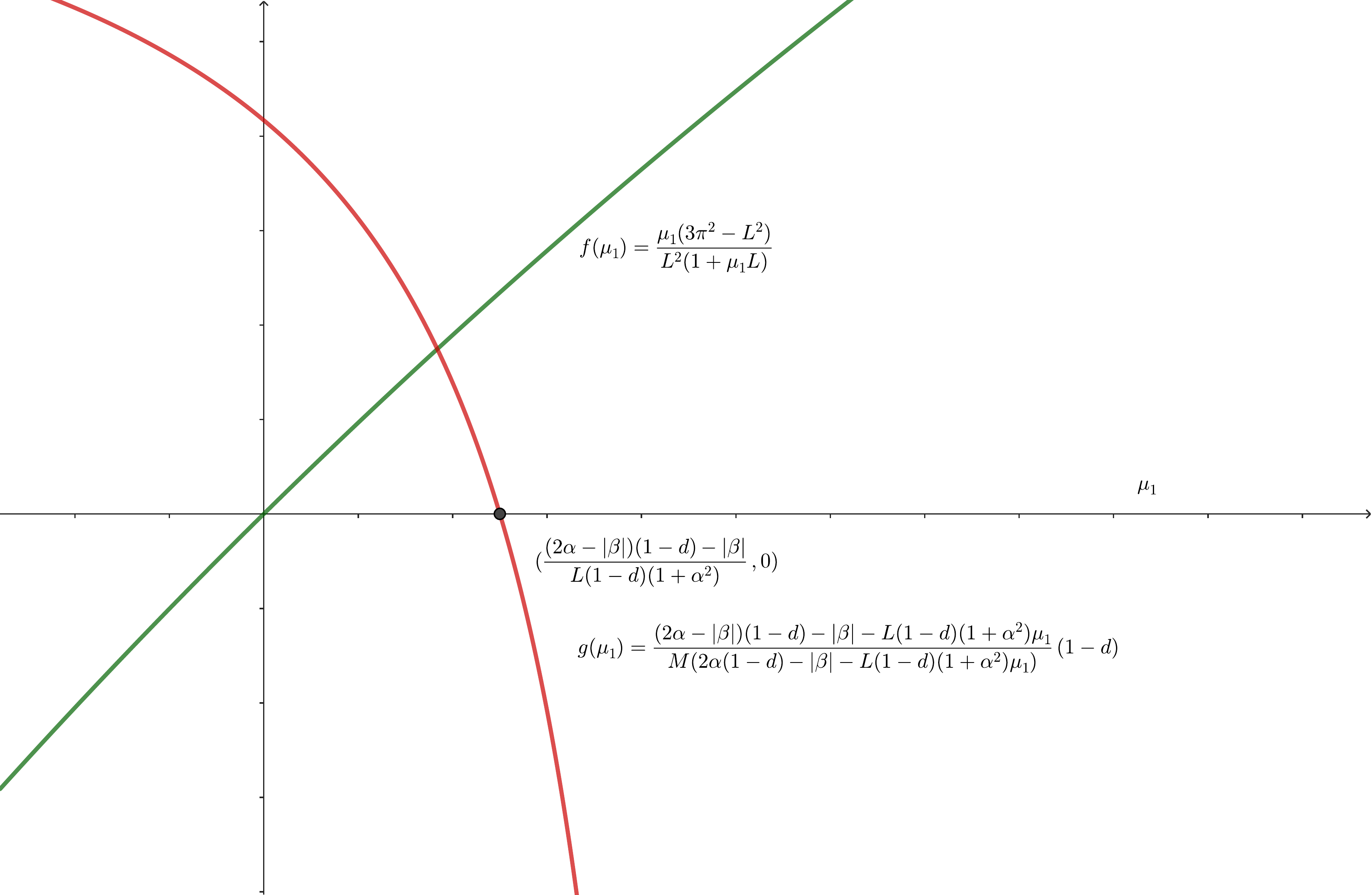}
	%%%call your figure name in the place "figurename.eps"
	\caption{Ilustration of Proposition \ref{llllll}}
	\label{fig1}
\end{figure}
\color{black}

\section{Concluding discussion}\label{sec4}
This article was concerned with the local well-posedness for the system \eqref{eq:KdV-KdV} and stabilization of the energy associated with the linearized KdV-KdV system posed on a bounded domain.  We proved the local well-posedness result by considering a linear combination of the damping mechanism and a time-varying delay term. Moreover, since we have the global solution associated with the linearized system, so, the energy method is used to show the exponential stabilization outcome for the linearized system.

\subsection{Further comments} The following remarks are worth mentioning.
%We will split them into two parts: comments on the full system \eqref{eq:KdV-KdV} in the first part, while the second one is devoted to the linear system.
%\subsubsection{Comments related to the full system}
\begin{enumerate}
\item The well-posedness finding is not proved directly. The main issue is due to the time-varying delay term that makes the associated operator for the system time-dependent. Therefore, we invoked the ideas introduced by Kato~\cite{Kato1970} to solve an abstract Cauchy problem of the ``hyperbolic'' type.

\item In \cite{Capistrano2019},  the authors showed the stabilization result when $\beta=0$. In this case, using the classical  compactness-uniqueness argument, they found a restrictive condition on the spatial length, that is, the stabilization follows if only if
$$L\notin\mathcal{N}:=\left\{  \frac{2\pi}{\sqrt{3}}\sqrt{k^{2}+kl+l^{2}}\,:k,\,l\,\in\mathbb{N}^{\ast}\right\}.$$
Additionally, in \cite{Capistrano2019}, the decay rate could not be characterized. In turn, due to the presence of the time-varying delay term in our problem, the restriction on the spatial length is  $L\in(0,\sqrt3\pi)$, which seems reasonable. Last but not least, the decay rate of the energy is explicitly provided contrary to \cite{Capistrano2019}. However, the drawback of our result is that it is only true for the linearized system.

\item It is noteworthy that the strategy used in \cite{Pazoto2008}, and more recently in \cite{Capistrano2019} ensures the global solution of the nonlinear system \eqref{eq:KdV-KdV} \textbf{without delay}. However, such a strategy can not be applied when a time-dependent delay occurs. This is due to the fact that in this case, the system is non-autonomous.  In addition to that, this strategy fails to provide the desired result (global existence of solutions) for the nonlinear system even if a constant delay $\tau(t)=h$ is considered. The reason is our operator $A$, defined by \eqref{eq:A}, has a transport part with nonhomogeneous boundary conditions given by the equation \eqref{eq:tr} and hence we can not expect to control the solution of the transport part in the space $H^{1/3}(0,1)$ in terms of the $L^2(0,1)$ norm of the initial data. Thus, for the full system \eqref{eq:KdV-KdV} with a constant delay $\tau(t)=h$, another approach needs to be applied. We discuss it in the last subsection of the work.
\color{black}

\item Naturally, it would be interesting to make a comparison between the KdV-KdV and the KdV models. Two important facts appear:
\begin{itemize}
\item The Lyapunov approach provides a direct way to deal with the nonlinear system KdV equation, as shown in \cite{Parada2022}. In this work, stability results for the KdV equation with time-varying delay are established using the same techniques. In comparison to our work, two KdV equations are coupled by the nonlinearities; thus the complexity of the problem suggests choosing a different Lyapunov functional and deals only with the linearized system.

\item Another interesting comparison is about the energy decay rate associated with the KdV and  KdV-KdV models, at least for the linear problem. In both cases, the explicit decay rate is shown.%, however, %due to the two coupled nonlinearities present in the KdV-KdV model, the decay rate shows that the energy, in this model, decays slower than the energy of the KdV model.
\end{itemize}
%\end{enumerate}
%\subsubsection{Comments about the linear system}\label{412}
%\begin{enumerate}
\item A calculation shows that taking $\mu_{1}$ and $\mu_{2}$ in Theorem \ref{th:Lyapunov0} such that
\begin{equation*}
	\mu_{1} < \min\left\{\frac{2\alpha-|\beta|}{L(1+\alpha^2)},\frac{(2\alpha-\lvert\beta\rvert)(1-d)-\lvert\beta\rvert}{L(1-d)(1+\alpha^2)}\right\}=\frac{(2\alpha-\lvert\beta\rvert)(1-d)-\lvert\beta\rvert}{L(1-d)(1+\alpha^2)}
\end{equation*}
and
\begin{equation*}
	\begin{aligned}
	\mu_{2} =& \min\left\{\frac{(2\alpha - \lvert\beta\rvert) -L(1+\alpha^2)\mu_1}{\lvert \beta \rvert },\frac{ (2\alpha-\beta)(1-d)-\lvert\beta\rvert-L(1-d)(1+\alpha^2)\mu_1}{\lvert \beta \rvert (1-d)}  \right\}\\
	=&\frac{ (2\alpha-\beta)(1-d)-\lvert\beta\rvert-L(1-d)(1+\alpha^2)\mu_1}{\lvert \beta \rvert (1-d)},
	\end{aligned}
\end{equation*}
implies that $\Psi_{\mu_1,\mu_2}$, given by \eqref{psi}, is negative definite provide that $|\alpha|<1$.

\vspace{0.2cm}

In fact, recall
\begin{equation*}
	\begin{aligned}
		\Psi_{\mu_1,\mu_2}
%		= \ &
%		\Phi_{\alpha,\beta} + \mu_1 L \begin{pmatrix} 1+\alpha^2 & -\alpha\beta \\ -\alpha\beta & \beta^2 \end{pmatrix} + \mu_2\lvert\beta\rvert \begin{pmatrix} 1 & 0 \\ 0 & 0 \end{pmatrix}\\
%		= \ &
%		\begin{pmatrix} -2\alpha + \lvert\beta\rvert & \beta \\ \beta & \lvert\beta\rvert(d-1) \end{pmatrix} +  \mu_1 L \begin{pmatrix} 1+\alpha^2 & -\alpha\beta \\ -\alpha\beta & \beta^2 \end{pmatrix} + \mu_2\lvert\beta\rvert \begin{pmatrix} 1 & 0 \\ 0 & 0 \end{pmatrix}\\
		= \ & \begin{pmatrix} -2\alpha + \lvert\beta\rvert+ \mu_1L(1+\alpha^2)+\mu_2\lvert\beta\rvert & \beta(1-L\mu_1 \alpha) \\  \beta(1-L\mu_1 \alpha)& \lvert\beta\rvert(d-1)+L\mu_1 \beta^2 \end{pmatrix}	=  \begin{pmatrix} a_{11} & a_{12}\\ a_{21} & a_{22}
		\end{pmatrix}.	
	\end{aligned}
\end{equation*}
In order to $\Psi_{\mu_1,\mu_2}$ be negative definite, the term $a_{11}$ must be negative,
$$
	-2\alpha + \lvert\beta\rvert + L\mu_1(1+\alpha^2) + \lvert\beta\rvert\mu_2 < 0 \Longleftrightarrow
	\mu_2 < \frac{(2\alpha - \lvert\beta\rvert) -L(1+\alpha^2)\mu_1}{\lvert \beta \rvert }
$$
with
\begin{equation*}
2\alpha - \lvert\beta\rvert -L(1+\alpha^2)\mu_1>0,
\end{equation*}
which implies that $\mu_1$ must satisfy
\begin{equation*}
\mu_1<\frac{2\alpha-|\beta|}{L(1+\alpha^2)}.
\end{equation*}
Moreover, we need  that
\begin{equation*}
	\det \Psi_{\mu_1,\mu_2}=\begin{vmatrix}
		-2\alpha+\lvert\beta\rvert + L\mu_1(\alpha^2+1) + \lvert\beta\rvert\mu_2 &  \beta(1-L\mu_1 \alpha) \\
		 \beta(1-L\mu_1 \alpha) &  \lvert\beta\rvert(d-1) + L\mu_1\beta^2
	\end{vmatrix} > 0.
\end{equation*}
Note that,
\begin{equation*}
	\begin{aligned}
		\det \Psi_{\mu_1,\mu_2}=&
 \lvert \beta \rvert \left[ (L \mu_1)^2\lvert \beta \rvert +L \mu_1  (1+\mu_2)\lvert \beta \rvert^{2}-L\mu_1 (\alpha^{2}+1)(1-d)\right.\\
&		\left. -\left((-2\alpha+\lvert \beta \rvert)(1-d)+\lvert \beta \rvert \mu_2 (1-d)+\lvert\beta\rvert\right) \right].
		\end{aligned}
\end{equation*}
Since $$(L\mu_1)^2\lvert\beta\rvert + L\mu_1\lvert\beta\rvert^2 (1+\mu_2)>0,$$
in order to the determinant of $\Psi_{\mu_1,\mu_2}$ be positive, we only need
\begin{equation*}
		-L\mu_1(1-d)(1+\alpha^2)-\left((-2\alpha+\lvert\beta\rvert)(1-d)+\lvert \beta \rvert \mu_2 (1-d)+\lvert\beta\rvert\right) = 0
	\end{equation*}
that is,
\begin{equation*}
	-L\mu_1(1-d)(1+\alpha^2)+(2\alpha-\lvert\beta\rvert)(1-d)-\lvert \beta \rvert \mu_2 (1-d)-\lvert\beta\rvert=0.
\end{equation*}
Thus, we have
\begin{equation*}
	\mu_2=\frac{ (2\alpha-\lvert\beta\rvert)(1-d)-\lvert\beta\rvert-L(1-d)(1+\alpha^2)\mu_1}{\lvert \beta \rvert (1-d)}
\end{equation*}
with
$$
	\mu_1 < \frac{(2\alpha-\lvert\beta\rvert)(1-d)-\lvert\beta\rvert}{L(1-d)(1+\alpha^2)}.
$$
\item From Theorem \ref{th:Lyapunov0} and item (4), it follows that when $L< \sqrt{3} \pi$ and by taking $\mu_1, \mu_2 >0$ so that $\mu_1 L <1$ and
\begin{equation*}
	\mu_1 < \frac{(2\alpha-\lvert\beta\rvert)(1-d)-\lvert\beta\rvert}{L(1-d)(1+\alpha^2)}, \quad
%	\end{equation*}
%	and
%	\begin{equation*}
	\mu_2 = \frac{ (2\alpha-\lvert\beta\rvert)(1-d)-\lvert\beta\rvert-L(1-d)(1+\alpha^2)\mu_1}{\lvert \beta \rvert (1-d)},
\end{equation*}
we reach that	$E(t) \leq \zeta E(0)e^{-\lambda t},$ for all $t \geq 0$ where
$$
\lambda \leq \min \left\lbrace
\frac{\mu_1(3\pi^2-L^2)}{L^2(1+\mu_1)}  , \frac{\mu_2(1-d)}{M(1+\mu_2)}
\right\rbrace \quad \text{and} \quad \zeta = \frac{1 + \max\{\mu_1 L,\mu_2\}}{1- \max\{\mu_1 L,\mu_2\}}.
$$
\end{enumerate}

\subsection{Open problems} There are some points to be raised.

\subsubsection{A time-varying delay feedback} The main difficulty when dealing with the problem \eqref{eq:KdV-KdV} is how to prove the global well-posedness. This is due to the lack of the $L^2$ a priori estimate. It is worth mentioning that, in this case, the semigroup theory or multipliers method cannot be applied, due to a restriction of ''controlling'' the solutions of the transport equation in specific norms. We believe that a variation of the approach introduced by Bona \textit{et al.} in \cite{BSS} can be adapted. However, this remains a promising research avenue, and the stabilization problem for the nonlinear system \eqref{eq:KdV-KdV} needs to be investigated.

\subsubsection{Variation of feedback-law} Considering two internal damping mechanisms and a linear combination of boundary damping and time-varying delay feedback, similar result of our work can be proved. Due to the restriction of the well-posedness problem, we cannot remove the boundary damping. However, an open problem is to remove one internal damping mechanism and make $\beta=0$. We believe that the Carleman estimate shown in \cite{BaGuPa} can be used to investigate all these cases.

\subsubsection{Optimal decay rate} Note that the Proposition \ref{llllll} gives the optimality of $\lambda$ for the stabilization problem related to the linear system associated with \eqref{eq:KdV-KdV}. In turn, it is still an open problem to obtain an optimal decay rate for both the linear and nonlinear problems without additional conditions for the parameters $\alpha$ and $\beta$.

\subsection*{Acknowledgment} The authors are grateful to the anonymous referees for the constructive comments that improved this work.

Capistrano--Filho was supported by CAPES grant numbers 88881.311964/2018-01 and  88881.520205/2020-01,  CNPq grant numbers 307808/2021-1 and  401003/2022-1,  MATHAMSUD grant 21-MATH-03 and Propesqi (UFPE). Mu\~{n}oz was supported by the P.hD. scholarship from FACEPE number IBPG-0909-1.01/20. This work is part of the Ph.D. thesis of Mu\~{n}oz at the Department of Mathematics of the UFPE. This work was done while the first author was visiting Virginia Tech. The author thanks the host institution for their warm hospitality.

\subsection*{Data availability statement} Data sharing does not apply to the current paper as no data were generated or analyzed during this study.

\subsection*{Conflict of Interest} The authors declare that they have no conflict of interest.


\begin{thebibliography}{100}
\bibitem{BaGuPa} J. A. Barcena-Petisco, S. Guerrero and A. F. Pazoto, \textit{Local null controllability of a model system for strong interaction between internal solitary waves}, Communications in Contemporary Mathematics, 24, 1-30 (2022).

\bibitem{Valein2019}  L. Baudouin, E. Cr\'epeau and J. Valein, \textit{Two approaches for the stabilization of nonlinear KdV equation with boundary time-delay feedback}, IEEE TAC 64:4, 1403--1414 (2019).

\bibitem{Bona2002}  J. L. Bona, M. Chen, J.-C. Saut, \textit{Boussinesq equations and other systems for small-amplitude long waves in nonlinear dispersive media. I. Derivation and linear theory}, J. Non. Sci. 12:4, 283--318  (2002).

\bibitem{Bona2004} J. L. Bona, M. Chen and J.-C. Saut, \textit{Boussinesq equations and other systems for small-amplitude long waves in nonlinear dispersive media. II. The nonlinear theory}, Nonlinearity 17, 925--952 (2004).

\bibitem{BSS} J. L. Bona,S.-M. Sun and B.-Y. Zhang, \textit{A nonhomogeneous boundary-value problem for the Korteweg-de Vries Equation on a finite domain}, Commun. PDEs, 8, 1391--1436 (2003).

\bibitem{bou} J. V. Boussinesq, \textit{Théorie générale des mouvements qui sont propagés dans un canal rectangulaire horizontal}, C. R. Acad. Sci. Paris 72, 755--759 (1871).

\bibitem{Boumediene2023} R. A. Capistrano--Filho, B. Chentouf, L. S. Sousa and V. H. Gonzalez Martinez,  \textit{Two stability results for the Kawahara equation with a time-delayed boundary control},  Z. Angew. Math. Phys. 74, 16 (2023).

\bibitem{Cerpa20XX}R. A. Capistrano--Filho, E. Cerpa, and F. A. Gallego,  \textit{Rapid exponential stabilization of a Boussinesq system of KdV–KdV Type},  Communications in Contemporary Mathematics, 25:03, 2150111 (2023).

 \bibitem{Capistrano2018} R. A. Capistrano--Filho and F. A. Gallego, \textit{Asymptotic behavior of Boussinesq system of KdV–KdV type}, Journal of Differential Equations 265:6, 2341--2374 (2018).

\bibitem{RoAn} R. A. Capistrano--Filho and A. Gomes, \textit{Global control aspects for long waves in nonlinear dispersive media}, ESAIM: Control, Optimisation, and Calculus of Variations, 29:7, 1--47 (2023).

\bibitem{Martinez2022} R. A. Capistrano--Filho and V. H. Gonzalez Martinez, \textit{Stabilization results for delayed fifth-order KdV-type equation in a bounded domain},  Mathematical Control and Related Fields, 14(1): 284--321 (2024).

\bibitem{Munoz2022} R. A. Capistrano--Filho, V. H. Gonzalez Martinez, and J.R. Muñoz, \textit{Stabilization of the Kawahara-Kadomtsev-Petviashvili equation with time-delayed feedback}, Proceedings of the Royal Society of Edinburgh: Section A Mathematics. Published online 2023:1-27. doi:10.1017/prm.2023.92.

\bibitem{Capistrano2019} R. A. Capistrano--Filho, A. F. Pazoto and L. Rosier, \textit{Control of Boussinesq system of KdV-KdV type on a bounded interval}, ESAIM: Control Optimization and Calculus Variations 25:58, 1--55 (2019).

\bibitem{bc2021} B. Chentouf, \textit{Well-posedness and exponential stability results for a nonlinear Kuramoto-Sivashinsky equation with a boundary time-delay,} Anal. Math. Phys., 11:144, (2021). %https://doi.org/10.1007/s13324-021-00578-1.

\bibitem{Chentouf22} B. Chentouf, \textit{Well-posedness and exponential stability of the Kawahara equation with time-delayed localized damping,} Math. Methods Appl. Sci. 45, 10312–10330 (2022).

\bibitem{d1} R. Datko, \emph{Not all feedback stabilized hyperbolic systems are robust with respect to small time delays in their feedbacks}, {SIAM J. Control Optim.}, 26, 697--713 (1988).

\bibitem{d2} R. Datko, \emph{Two examples of ill-posedness with respect to time delays revisited}, {IEEE Trans. Automat. Control}, 42, 511--515 (1997).

\bibitem{d3} R. Datko, J. Lagnese, and M.P. Polis, \emph{An example on the effect of time delays in boundary feedback stabilization of wave equations}, {SIAM J. Control Optim.}, 24, 152--156 (1986).

\bibitem{Micu2009}  S. Micu, J. H. Ortega, L. Rosier and B.-Y. Zhang, \textit{Control and stabilization of a family of Boussinesq systems}, Discrete Contin. Dyn. Syst., 24, 273--313 (2009).

\bibitem{Nicaise2006}S. Nicaise and C. Pignotti,  \textit{Stability and instability results of the wave equation with a delay term in the boundary or internal feedbacks}, SIAM J. Control Optim.  45:5, 1561--1585 (2006).

\bibitem{np1}
E. E. Fridman, S. Nicaise and J. Valein, \emph{Stabilization of second order evolution equations with unbounded feedback with time-dependent delay}, SIAM J. Control Optim., 48(8), 5028--5052 (2010).

 \bibitem{np2}
S. Nicaise and C. Pignotti, \emph{Interior feedback stabilization of wave equations with time-dependent delay}, {Electronic Journal of Differential Equations}, vol. 2011, no. 41, 1--20 (2011).

\bibitem{np3} S. Nicaise, C. Pignotti and J. Valein, \emph{Exponential stability of the wave equations with boundary time-varying delays}, {Discrete Contin. Dyn. Syst. Ser. S}, 4, 693--722 (2011).

\bibitem{Nicaise2009} S. Nicaise, J. Valein, and E. Fridman,  \textit{Stability of the heat and of the wave equations with boundary time-varying delays}, Discrete Continuous Dynamical Systems-S., 2(3):559--581 (2009).

\bibitem{Pazy}  A. Pazy, \textit{Semigroups of Linear Operators and Applications to Partial Differential Equations}, volume 44 of Applied Math. Sciences, Springer-Verlag, New York, 1983.

\bibitem{Parada2022} H. Parada, C. Timimoun, J. Valein. \textit{Stability results for the KdV equation with time-varying delay}, Syst. Control Lett., 177, (2023),
105547.

\bibitem{Pazoto2008}  A. F. Pazoto and L. Rosier, \textit{Stabilization of a Boussinesq system of KdV–KdV type}, Syst. Control Lett. 57, 595--601 (2008).

\bibitem{Kato1970}  T. Kato, \textit{Linear evolution equations of “hyperbolic” type}. J. Fac. Sci. Univ. Tokyo Sect. I, 17, 241--258 (1970).

\bibitem{Kato1975} T. Kato, \textit{Quasi-linear equations of evolution, with applications to partial differential equations}, In Spectral theory and differential equations,  Springer, 25--70 (1975).

\bibitem{Valein20XX} J. Valein, \textit{On the asymptotic stability of the Korteweg-de Vries equation with time-delayed internal feedback}, Mathematical Control \& Related Fields, 12:3, 667--694 (2022).

    \bibitem {xyl} G. Q. Xu, S. P. Yung, and L. K. Li, \textit{Stabilization of wave systems with input delay in the boundary control}, ESAIM Control Optim. Calc. Var., 12, 770--785 (2006).

\end{thebibliography}
\end{document}